\documentclass[11pt,reqno]{amsart}%
\usepackage[latin1]{inputenc}
\usepackage{amsfonts,amssymb,latexsym}
\usepackage{amsmath,amsthm}
\usepackage{enumerate}
\usepackage[dvips]{color,graphicx}
\usepackage{amsbsy}
\usepackage{amsfonts}
\usepackage{graphicx}
\usepackage{color}
\usepackage{amsmath}
\usepackage{amssymb}%
\usepackage{amsthm}
\usepackage[dvips]{graphicx}
\usepackage{graphicx} 
\usepackage{subfigure}
\usepackage{pst-eucl}
\usepackage[english]{babel}
\usepackage{amsmath,amssymb,graphics,mathrsfs}
\usepackage{amsmath,amssymb,latexsym}
\usepackage{graphicx,color}
\usepackage[T1]{fontenc}
\usepackage[active]{srcltx}
\usepackage{multicol}
\usepackage[latin1]{inputenc}
\usepackage{pst-all}
\usepackage{enumerate}
\usepackage{pstricks}
\usepackage{pstricks-add}
\usepackage{setspace}
\usepackage{soul}
\usepackage{cancel}
\usepackage{nonfloat}
\usepackage[left=3.8cm,top=2.8cm,right=2.2cm,bottom=3.1cm]{geometry}

\usepackage[colorlinks=true,citecolor=red,linkcolor=blue,urlcolor=RubineRed,pdfpagetransition=Blinds,pdftoolbar=false,pdfmenubar=false]{hyperref}

\newtheorem{theorem}{Theorem}[section]
\newtheorem{definition}[theorem]{Definition}
\newtheorem{proposition}[theorem]{Proposition}
\newtheorem{corollary}[theorem]{Corollary}
\newtheorem{lemma}[theorem]{Lemma}
\newtheorem{remark}{Remark}[section]

\numberwithin{equation}{section}

\begin{document}
\title[ ]{Optimal bilinear control restricted to the three-dimensional chemo-repulsion model with potential production}
\author[F. Guill\'en-Gonz\'alez]{Francisco Guill\'en-Gonz\'alez}
\address[F. Guill\'en-Gonz\'alez]{Universidad de Sevilla, Dpto. de Ecuaciones Diferenciales y An\'alisis Num\'erico and IMUS, C/Tarfia, S/N, 41012, Sevilla, Spain.}\email{guillen@us.es}
\author[E. Mallea-Zepeda]{Exequiel Mallea-Zepeda}
\address[E. Mallea-Zepeda]{Universidad  Cat\'olica del Norte, Departamento de
Matem\'{a}ticas, Av. Angamos, 0610, Antofagasta, Chile.} \email{emallea@ucn.cl}
\author[M. A. Rodr\'iguez-Bellido]{Mar\'ia A. Rodr\'iguez-Bellido }
\address[M.A. Rodr\'iguez-Bellido]{Universidad de Sevilla, Dpto. de Ecuaciones Diferenciales y An\'alisis Num\'erico and IMUS, C/Tarfia, S/N, 41012, Sevilla, Spain.}\email{angeles@us.es}
\author[E.J. Villamizar-Roa]{\'Elder J. Villamizar-Roa}
\address[E.J. Villamizar-Roa]{Universidad Industrial de Santander, Escuela de
Matem\'{a}ticas, A.A. 678, Bucaramanga, Colombia.} \email{jvillami@uis.edu.co}
\date{\today}

\keywords{Chemo-repulsion, nonlinear production, weak and strong solutions, optimal control.} \subjclass[2020]{35A01; 35Q92; 49J20; 49K20; 93C10}

\begin{abstract}
In this paper we study the following three-dimensional parabolic-parabolic chemo-repulsion model with potential production, logistic reaction and bilinear control, defined in $Q=(0,T)\times\Omega$:
\begin{equation*}\label{eq0}
\left\{
\begin{array}{rcl}
\partial_tu-\Delta u&=&\nabla\cdot(u\nabla v)+r\,u-\mu\, u^p,\\
\partial_tv-\Delta v+v&=&u^p+f\,v\, 1_{\Omega_c},
\end{array}
\right.
\end{equation*}
where $1< p<+\infty$, $r,\mu\geq 0$, and $f=f(t,x)$ is the control function acting  on a  subdomain $(0,T)\times \Omega_c $,  with $\Omega_c\subseteq\Omega$. 
This system is endowed with initial and non-flux boundary conditions. 
We prove the existence of global weak solutions of this controlled problem when $f\in L^{5/2}(0,T;L^{5/2}(\Omega_c))$, 
analyzing the role  of the diffusion and the logistic terms to get  energy estimates. In particular,  the logistic competition term $\mu \, u^p$ is  necessary only for $p>5/3$.
Secondly, if $f\in L^{5/2}(0,T;L^{5/2+}(\Omega_c))$, any weak solution $(u,v)$ satisfying the regularity criterion  $u\in L^{5p/2}(Q)\cap L^{10/3}(Q)$ is in fact more regular, arriving in particular to $u,\nabla v\in L^5(Q)$ for $p\le 2$ and $u,\nabla v\in L^{5(p-1)}(Q)$ for $p> 2$
which is the critical regularity to solve a related optimal bilinear control problem.
 In fact, this setting let us to prove the existence of global optimal solutions, and  the differentiability of the control-to-state mapping via the Implicit Function Theorem in Banach spaces. Then, we can identify the gradient of the (reduced) cost with respect to the control solving the adjoint problem
 by duality.  
 In particular, we derive  first-order necessary optimality conditions for local optimal solutions. 
\end{abstract}
\maketitle

\section{Introduction}
Chemotaxis involves the  directional movement of cells or living organisms influenced by the concentration of the chemical signal substance. This phenomenon plays an essential role in several biological processes such as wound healing, cancer invasion, and avoidance of predators \cite{Isenbach}. 
This motion can be towards a higher (attractive) or lower (repulsive) concentration of the chemical stimuli. In this paper we are interested in the repulsive chemotaxis scenario, in which the presence
of living organisms (cells for instance) produce chemical substance, including or not a logistic growth of organisms. Explicitly, considering $Q:=(0,T)\times\Omega,$ with $\Omega\subset\mathbb{R}^3$ being a bounded domain and $(0,T),$ $T>0,$ a time interval, we consider the following PDE systems in $Q$:\\
{\it Non-logistic system:}
\begin{equation}\label{eq1}
\left\{
\begin{array}{rcl}
\partial_tu-\Delta u&=&\nabla\cdot(u\nabla v),\\
\partial_tv-\Delta v+v&=&u^p+fv\, 1_{\Omega_c}.
\end{array}
\right.
\end{equation}
{\it Logistic system:}
 \begin{equation}\label{eq1b}
\left\{
\begin{array}{rcl}
\partial_tu-\Delta u&=&\nabla\cdot(u\nabla v)+ru-\mu u^p,\\
\partial_tv-\Delta v+v&=&u^p+fv\, 1_{\Omega_c}.
\end{array}
\right.
\end{equation}
In (\ref{eq1}) and (\ref{eq1b}), the unknowns are $u(t,x)\ge0$ and  $v(t,x)\ge 0$ denoting the cell density
 and the chemical concentration,
 respectively.  The second order nonlinear term $\nabla\cdot(u\nabla v)$ 
 describes the repulsion mechanism. The opposite phenomenon, that is, the attractive chemotaxis, corresponds to the case of $-\nabla\cdot(u\nabla v)$ (see the surveys \cite{Bellomo,Hillen,Lankeit}). 
On the other hand, the term $u^p,$ $p>1,$ in the $v$-equation, corresponds to the nonlinear production of chemical signal by cells, and $fv\, 1_{\Omega_c}$ is a reaction term modelling a bilinear control where the control $f$ acts as a proliferation or degradation coefficient of the chemical substance, on a localized control region $\Omega_c\subseteq \Omega$ (hereafter $1_{\Omega_c}$ denotes the characteristic function on $\Omega_c$). This bilinear control allows to maintain nonnegativity of state variables $u$ and $v$ without restrictions on the sign of the control $f$.  
 Thus, the sign of controls can be  seen as the injection or extraction of chemical substances. Finally, in (\ref{eq1b}), $ru-\mu u^p,$ $r,\mu>0,$  
represents a proliferation term following a logistic regime.
 Systems (\ref{eq1}), (\ref{eq1b}) are completed with the initial and non-flux boundary conditions:
 \begin{equation}\label{eq2}
\left\{
\begin{array}{rcl}
u(0)=u_0\ge0,\ v(0)=v_0\ge 0\ \mbox{ in }\ \Omega,\\
\partial_{\bf n} u=0,\ \partial_{\bf n} v=0\ \mbox{ on }\ (0,T)\times\partial\Omega.
\end{array}
\right.
\end{equation}
 Hereafter, we denote $\partial_{\bf n} w=\nabla w\cdot {\bf n}$ where ${\bf n}$ is the outward unit normal vector on $\partial\Omega$.\\

With respect to uncontrolled problems ($f \equiv 0$), although the chemorepulsion system seems to be more favorable to obtain regularity properties  than the chemoatractive  one, 
 it is still far from being completely understood if there exists blow-up or not for three-dimensional domains. In fact, 
 the existence and uniqueness of classical global-in-time solutions in bi-dimensional (or one-dimensional) bounded domains for the problem (\ref{eq1}) without control ($f \equiv 0$) and linear production ($p=1$) was proved in   \cite{cieslak}, obtaining also  
 the existence of global weak solutions for 3D  domains. In addition,
 the existence of global weak solutions for 
 (\ref{eq1}) in 3D bounded domains  and $f \equiv 0$ was addressed either  for quadratic production ($p=2$) in \cite{Diego1}, or for the subquadratic case $p\in(1,2)$ in \cite{Diego3}.
 However, for any $p\ge 1$,  it is unclear whether regular bounded solutions exist globally in time. Assuming more favorable conditions, either superlinear diffusion \cite{Freitag} or  sufficiently weak chemotactic sensitivity  \cite{tao}, 
  the existence of  regular global in time and bounded solutions holds.
 \\ 

With respect to the controlled problems ($f \not\equiv 0$) in 2D domains, the  existence and uniqueness of global strong solutions for problem (\ref{eq1})  
 is proved in \cite{Exequiel2D} for  $p=1$ and 
for subquadratic production $1<p\leq 2$ in \cite{EFE2D}. 
More recently, the existence of global weak solutions for problem (\ref{eq1b}) with linear production ($p=1$)  and quadratic logistic growth $r u-\mu u^2$ was proved in \cite{Guillen4}.
On the other hand, with respect to the controlled problem (\ref{eq1})  in 3D domains and linear production $p=1$,  the existence of global weak solutions   and  a regularity criterion  allowing to obtain unique strong solution was established in  \cite{Exequiel3D}.
 
 \
 
  The first motivation of this work is to understand  the effect of the nonlinear production in the  problems with bilinear control,  in 3D domains.
In this sense, we will prove the existence of global weak solutions for the no-logistic problem  (\ref{eq1}), provided $1\leq p\leq 5/3$. For the complementary case $p>5/3$, it is necessary to add the logistic term $r u-\mu u^p$ in order to control  the energy associated to the system, and deduce the existence of weak solutions of (\ref{eq1b}). 
  Once the existence of global weak solutions is known, 
 we establish a {\it Serrin-type} regularity criterion for both systems (\ref{eq1}) and  (\ref{eq1b}),  
 which makes weak solutions 
  become strong solutions. Indeed, we prove that if $f\in L^{5/2}(0,T;L^{5/2+}(\Omega_c))$ 
  ($L^{5/2+}$ means $L^{5/2+\delta}$ for $\delta$ small enough), then any weak solution $(u,v)$ satisfying the regularity criterion $u\in L^{5p/2}(Q)\cap L^{10/3}(Q)$ is more regular and unique. Note that we have chosen the competition part with the same power $p$ as in the production term, because this is the least regular case allowing the existence of solutions, 
and consequently, the case of greatest mathematical interest; however, our results could be adapted for other cases where the power in the logistic term is greater than the power in the production. \\

Our interest in the existence 
and regularity of  solutions of (\ref{eq1}) or (\ref{eq1b})
is strongly motivated by our intention of analyzing  optimal control problems constrained
 to  strong solutions of these chemorepulsion systems. 
In  bidimensional bounded domains,  optimal control problems for (\ref{eq1}) have been analyzed, either for linear production ($p=1$) in \cite{Exequiel2D} or for subquadratic production ($p\in(1,2]$)  in \cite{EFE2D}. 
Even in 2D domains, an optimal control problem for 
 (\ref{eq1b}) with logistic term $ru-\mu u^2$ and linear production ($p=1$) was analyzed in \cite{Guillen4} (see also \cite{Hernandez} for an optimal control problem subjected to a Lotka-Volterra competition model with chemorepulsion). 
In the context of 3D-bounded domains, up our knowledge the only result dealing with a control problem is obtained in \cite{Exequiel3D} subjected to the problem (\ref{eq1}) with linear production ($p=1$). Therefore there is a gap in the case of the 3D optimal control problem constrained
to chemorepulsion models (\ref{eq1}) and  (\ref{eq1b}) for nonlinear production $p>1.$ \\

Based on the above, our second motivation is to study optimal control problems with state equations given by (\ref{eq1}) or (\ref{eq1b}) and $p>1$. In fact, we prove the existence of global optimal solutions via a minimizing sequence argument. 
Then,  first-order necessary optimality conditions for local optimal solutions are obtained,  where, instead of using a generic Lagrange multipliers argument as addressed in \cite{Guillen4,Andre_SICOM,Exequiel2D,Exequiel3D,EFE2D, Hernandez,linares, Lorca, rodriguez_rueda}, 
we compute the derivative of the state with respect to the control via  
the Implicit Function Theorem.
Indeed, 
we prove that the control-to-state operator is locally well-defined around a sufficiently regular solution, and it is Fr\'echet differentiable. 
With this in hand, 
we can compute the derivative of the reduced cost (defined only with respect to the control)
by using the adjoint problem to identify this derivative in an explicit form. Moreover, 
 the well-posedness of the adjoint system is deduced by a duality argument. 
As consequence, we directly deduce  the first-order necessary optimality conditions.
Notice that, the explicit calculus of the derivative of the reduced cost  functional could be  an appropriate strategy to propose numerical schemes 
 based on gradient descent methods.\\
 
 
 We summarize the main contributions of the paper as follow: 
 \begin{itemize}
 \item We prove the existence of global-in-time weak solutions for a three-dimensional highly nonlinear chemorepulsion system analyzing carefully the loss of regularity produced by the
nonlinear production, and its compensation, at least when $1\leq  p \leq 5/3$, by the diffusion.  In fact, we  prove existence
of global weak solutions for the non-logistic problem (\ref{eq1}) provided $1\leq  p \leq 5/3,$ while for the
complementary case $p > 5/3$  is necessary to add the logistic term. The cornerstone of this proof is to arrive at the  energy inequalities (\ref{energy0}) and (\ref{ps-120}) for non-logistic and logistic cases, respectively.
 \item We establish a regularity criterion which makes weak solutions become more regular. The regularity criterion is obtained through a not standard bootstrapping argument based on a linear diffusion-convection-reaction equation with non-regular coefficients for the convection and reaction terms. In particular, regularity $u\in L^{10/3}(Q)$ will be required to prove, by duality,  uniqueness  between any weak solution and the weak solution satisfying the regularity  criterion.
\item The regularity criterion is established with the minimum regularity of the control $f\in L^{5/2}(L^{5/2+}(\Omega_c))$ and the state variables $(u,v)$, defined in suitable Banach spaces $X_u$ and $Y_u$  (see Subsection~\ref{section2.3}), 
 that implies the well-posedness of the linearized system (see Lemma~\ref{corol-A}).
 \item We  can identify the gradient of the (reduced) cost functional with respect to the control, owing to the
differentiability of the control-to-state operator (via the Implicit Function Theorem) and the adjoint problem (solved by a duality argument). This fact let us to derive first-order necessary optimality conditions for local optimal solutions. On the other hand,  the identification of the gradient can be an appropriate strategy to propose numerical schemes based on gradient descent methods, which have not been  explored properly in the context of chemotaxis problems.
 \end{itemize}

The layout of this paper is as follows: In Section 2, we present the main contributions of this paper.
 In Section 3, we first  discuss some formal estimates in order to illustrate the difficulties for obtaining an energy inequality, highlighting the role of the nonlinear production. After, we prove 
 the existence of weak solutions  
in both cases, logistic and non-logistic problems; and establish
 a regularity criterion through which weak solutions become more regular solutions,
  satisfying, in particular, uniqueness. 
In Section 4, we prove the existence of optimal solutions restricted to state system (\ref{eq1}) or (\ref{eq1b}), 
and derive first-order necessary optimality conditions, analyzing the differentiability of the control-to-state operator, via the Implicit Function Theorem, and solving the adjoint problem by a duality argument. Finally, we include an appendix addressing the existence and regularity of some auxiliary linear problems, which will be applied to prove the well-posedness of the linearized system, an essential hypothesis to apply the Implicit Function Theorem.

\section{Main contributions}
\subsection{Basic notation.}
We start this section by recalling some basic notations which will be used later.
Let $\Omega\subset \mathbb{R}^3$ be a bounded Lipschitz domain with boundary $\partial\Omega$ of class $C^{2,1}$. 
 We use the notation standard for Lebesgue and Sobolev spaces. The norm in $L^s(\Omega),$ $1\leq s\leq \infty,$ $s\neq 2,$ will be denoted by $\Vert \cdot\Vert_{L^s},$ and if $s=2,$ we simply denote the $L^2$-norm by $\Vert\cdot\Vert$ and the respective inner product by $(\cdot,\cdot)$. We consider the usual Sobolev spaces $W^{m,s}(\Omega)=\{u\in L^s(\Omega)\,:\, \|\partial^\alpha u\|_{L^s}<\infty,\, \forall |\alpha|\le m\}$, with norm denoted by $\|\cdot\|_{W^{m,s}}.$ When $s=2$, we write $H^m(\Omega):=W^{m,2}(\Omega)$ and we denote the respective norm by $\|\cdot\|_{H^m}.$ For a Banach space $X$ and a final time $T>0$, the Bochner spaces $L^s(0,T;X)$ will be denoted by $L^s(X)$ and its norm by $\Vert \cdot \Vert_{L^s(X)};$ in particular, $L^s(0,T;L^s(\Omega))$ is simply denoted by $L^s(Q)$, and $L^q(0,T;L^s(\Omega))$  and $L^q(0,T;W^{m,s}(\Omega))$ are reduced to $L^q(L^s)$ and $L^q(W^{m,s})$, respectively.  Also, $C(X)$ will denote the space $C([0,T];X)$ with the uniform metric in time. 
The topological dual space of a Banach space $X$ will be denoted by $X'$, and the duality for a pair $X$ and $X'$ by $\langle\cdot,\cdot\rangle_{X'}$ or simply by $\langle\cdot,\cdot\rangle$ unless this leads to ambiguity. Moreover, some letters like $C,K,...$ will denote 
positive constants, independent of state $(u,v)$ and control $f$, but their values may change from line to line. 
We recall the following equivalent norms in $H^1(\Omega)$ and $H^2(\Omega)$
(see \cite{necas} for more details):
\begin{eqnarray}
\|u\|^2_{H^1}&\equiv& \|\nabla u\|^2+\left(\int_\Omega u\right)^2\quad \forall \,u\in H^1(\Omega),\label{norm-1}\\
\|u\|^2_{H^2}&\equiv& \|\Delta u\|^2+\left(\int_\Omega u\right)^2\quad \forall \, u\in H^2_{\bf n}(\Omega),\label{norm-2}
\end{eqnarray}
where $H^2_{\bf n}(\Omega)=\{u\in H^2(\Omega): \partial_{\bf n} u\vert_{\partial\Omega}=0\}.$ In order to analyze the existence theory, frequently we will
use the (weak) energy space
$$W_2=L ^\infty(L^2)\cap L^2(H^1),$$
as well as, the following (strong) Banach space
\begin{equation}\label{spaceX}
X_q:=\left\{u\in C([0,T];W^{2-2/q,q})\cap L^p(W^{2,q})\,:\, \partial_tu\in L^q(Q)\right\}\ (q>1).
\end{equation}
\begin{lemma}[\cite{Andre_SICOM} Lemmas 2.5 and 2.6]\label{An}  
The following continuous embeddings of the strong space $X_q$ into $L^s(L^r)$ spaces hold:
\begin{enumerate}
\item If $q\in [1,5/2),$ then $X_q\subset L^{5q/(5-2q)}(Q);$
\item If $q=5/2,$ then  $X_{5/2}\subset L^{\infty}(L^r)$ for all $r\in [1,\infty)$ (that we called $L^{\infty}(L^{\infty-})$);
\item If $q>5/2,$ then $X_q\subset L^\infty(Q);$
\item If $q\in(1,5)$ and $w\in X_q,$ then $\nabla w\in L^{5q/(5-q)}(Q)$. In particular, if  $w\in X_{5/2}$ then $\nabla w\in L^{5}(Q)$.
\end{enumerate}
\end{lemma}

\subsection{Results of existence of weak solutions and regularity}\label{2.2}
\begin{definition}[Weak solutions]\label{weak}
Let $f\in L^{5/2}(Q_c):=L^{5/2}(L^{5/2}(\Omega_c))$, 
$(u_0,v_0)\in L^{p}(\Omega)\times H^{1}(\Omega)$, with
$u_0\ge0$ and $v_0\ge 0$ a.e. in $\Omega$. 
A pair $(u,v)$ is called weak solution in $[0,T]$ of system
(\ref{eq1})-(\ref{eq2}) or  (\ref{eq1b})-(\ref{eq2}),
if $u\ge 0$ and $v\ge 0$ a.e. in $Q$,
\begin{eqnarray} \label{st-11}
&& u\in L^{\infty}(L^p)\cap L^{{5p}/{3}}(Q),\quad v\in L^\infty(H^1)\cap L^2(H^2_{\bf n}),\\
&& u\in L^{2p-1}(Q)\ \mbox{in  the logistic problem,}\label{stt1}
\end{eqnarray}
\begin{equation} \label{st-11-bis}
\nabla u\in L^{\gamma(p)}(Q);
\quad \gamma(p)=
\left\{
\begin{array}{ll}
5p/(3+p)& \hbox{when $1<p\le 2$,}\\
\noalign{\vspace{-2ex}}\\
25p/(18+5p)& \hbox{when $2<p<12/5$,} \\
\noalign{\vspace{-2ex}}\\
2 & \hbox{when $p\ge 12/5$,} 
\end{array}
\right.
\end{equation}
satisfying the variational formulation  of the $u$-equation
\begin{equation}\label{var2}
\int_0^T\langle \partial_tu,\overline{u}\rangle+\int_0^T\int_\Omega \nabla u\cdot \nabla  \overline{u}+\int_0^T\int_\Omega u\nabla v\cdot\nabla \overline{u}=\int_0^T\int_\Omega (r\,u-\mu \, u^p)\overline{u},
\end{equation}
for all $\overline{u}\in L^{5/2}(Q)
$ with $\nabla \overline{u} \in L^{10}(Q)$, 
 the $v$-equation 
  holds a.e $(t,x)\in Q$ and the initial conditions in (\ref{eq2}). 
\end{definition}
\begin{remark}
Looking at \eqref{st-11-bis}, there is a gap in the regularity of $\nabla u$ for $p\le 2$ and  $p>2$. 
\end{remark}
\begin{remark} \label{re:t-deriv} (Regularity of the time derivatives). 
From \eqref{st-11} and \eqref{st-11-bis}, one has the minimal regularity $\nabla u\in L ^{5/4}(Q)$,  
and $u\nabla v\in L ^{10/9}(Q)$, hence in particular
one has 
$$\partial_t u\in (L^{10}(W^{1,10})\cap L^{5/2}(Q))'
\quad \hbox{and} \quad
\partial_tv\in L^{5/3}(Q).
$$ 
\end{remark}
\begin{remark} (Sense of the initial and boundary conditions).
From regularity of $(u,v)$ given in \eqref{st-11} and of $(\partial_t u, \partial_t v)$ given in Remark \ref{re:t-deriv}, it holds the weak continuity
$$u \in C^0_w([0,T];L^p)\quad \hbox{and}\quad v \in C^0_w([0,T];H^1).
$$
 In particular, the initial conditions $(u(0),v(0))=(u_0,v_0)$ has sense in $L^p(\Omega)\times H^1(\Omega)$.
The boundary condition $\partial_{\bf n} u=0,$ on $(0,T)\times \partial\Omega$ jointly to the $u$-equation are satisfied in the variational sense given in \eqref{var2}, while $\partial_{\bf n} v=0$ on $(0,T)\times \partial\Omega$ is satisfied  because $v(t,\cdot)\in H^2_{\bf n}(\Omega)$ a.e. $t\in (0,T)$.
\end{remark}
\begin{theorem}[Existence of weak solutions]\label{teo1} 
Assume that  $f\in L^{5/2}(Q_c)$, $(u_0,v_0)\in L^{p}(\Omega)\times H^{1}(\Omega)$, with
$u_0\ge0$ and $v_0\ge 0$ a.e. in $\Omega.$ 
\begin{enumerate} \item If $1<p\le 5/3$, then there exists a weak solution of system (\ref{eq1}).
\item If $p>1$, then there exists a weak solution of system (\ref{eq1b}).
\end{enumerate}
\end{theorem}
The proof of Theorem \ref{teo1} is given in Subsection \ref{Se:3.1}.\\


%

Knowing the existence of global weak solutions, we establish a regularity criterion 
implying that weak solutions  
 become more regular solutions (which for simplicity we will called strong solutions). These strong solutions will give the adequate framework to study  an associated optimal control problem. In order to introduce strong solutions,  consider the following $p$-dependent coefficients
$$\alpha(p)=\max\{3,3(p-1)\},
\quad 
\beta(p)=\max\{5,5(p-1)\}$$
and the Banach space
  \begin{equation} \label{xu}
X:=\{u: \, u \in L^{\infty}(L^{\alpha(p)}) \cap L^{\beta(p)}(Q), \, \nabla u \in L^{2}(Q), \, 
\partial_t u \in L^{2}((H^{1})') \}.
\end{equation}

\begin{theorem}[Regularity criterion and existence of strong solutions]\label{strong}
Let $(u,v)$ be any weak solution of system (\ref{eq1}) or (\ref{eq1b}).
If in addition, 
$(u_0,v_0)\in L^{\alpha(p)}(\Omega)\times W^{6/5,5/2}(\Omega)$, 
  \begin{equation}\label{reg-crit}
f\in L^{5/2}(L^{5/2+}(\Omega_c)) \quad \hbox{and}\quad u\in L^{5p/2}(Q)\cap L^{10/3}(Q),
\end{equation}
then $(u,v)\in  X \times X_{5/2}$ (that we called strong solutions).  In addition, the strong solution is unique, and there exists a positive constant 
$$\widehat{K}:=\widehat{K}(T,
\|u_0\|_{L^{\alpha(p)}},
\|v_0\|_{W^{6/5,5/2}},\|f\|_{L^{5/2}(L^{5/2+}(\Omega_c))}, \|u^p\|_{L^{5/2}(Q)})$$ such that
\begin{equation}\label{bound_sol}
\|(u,v)\|_{ X\times X_{5/2}}\le \widehat{K}.
\end{equation}
\end{theorem}
The proof of Theorem \ref{strong} is given in Subsection \ref{seccionCriterium}.
\subsection{A related optimal control problem}\label{section2.3}
Knowing the existence  of weak solutions of system (\ref{eq1}) or (\ref{eq1b}) and the established regularity criterion \eqref{reg-crit},  we analyze a related optimal control problem.  Without loss of generality, we focus the analysis of the control problem considering the system with logistic term (\ref{eq1b}).
In any case, in order to study optimality conditions, the corresponding  linearized and adjoint problems must be solved in a suitable setting linked to the functional framework of the control-state problem (\ref{eq1b}). 
For this, we first consider the Sobolev space
$$
Z:=\{\varphi\in L^{5\alpha(p)/(3\alpha(p)-3)}(Q):\ \nabla\varphi\in L^{5\alpha(p)/(4\alpha(p)-3)}(Q)^3\},
$$
endowed with the norm $\| \varphi  \|_Z=\| \varphi \|_{L^{5\alpha(p)/(3\alpha(p)-3)}}+ 
\| \nabla\varphi \|_{L^{5\alpha(p)/(4\alpha(p)-3)}(Q)}$. Then, we define 
the vector space
\begin{eqnarray*}
Y_u:=\{g^0-\nabla\cdot{\bf g}^1 \ \hbox{in $Z'$}:\, g^0\in L^{5\alpha(p)/(2\alpha(p)+3)}(Q),\ {\bf g}^1\in L^{5\alpha(p)/(\alpha(p)+3)}(Q)^3\}.
\end{eqnarray*} 
In fact, $Y_u\subset Z'$ because the map $\varphi\in Z \mapsto \int_0^T\int_\Omega g^0\varphi+\int_0^T\int_\Omega {\bf g}^1\cdot\nabla \varphi \in \mathbb{R}$ is linear and continuous.
  On the other hand, by using the Hahn-Banach Theorem (mimicking the Proposition 8.14 in \cite{Brezis}), each element of $Z^\prime$ can be identified with elements of  $Y_u$; thus, 
  $$Y_u=Z'.
  $$
\begin{proposition}\label{ident}
Let $F\in Z'.$ Then, there exists 
$$(g^0,{\bf g}^1)\in L^{5\alpha(p)/(2\alpha(p)+3)}(Q)\times (L^{5\alpha(p)/(\alpha(p)+3)}(Q))^3$$ such that
$$\langle F,\varphi\rangle=
\int_0^T\int_\Omega g^0\varphi+\int_0^T\int_\Omega {\bf g}^1\cdot\nabla \varphi,
\quad \forall\varphi\in Z.
$$
In fact, $F=g^0-\nabla\cdot {\bf g}^1$ in $Z'$.
\end{proposition}
\begin{proof}
Consider the product space $E=L^{5\alpha(p)/(3\alpha(p)-3)}(Q)\times (L^{5\alpha(p)/(4\alpha(p)-3)}(Q))^3$ endowed with the norm 
$\Vert (h^0,{\bf h}^1)\Vert_E=\Vert h^0\Vert_{L^{5\alpha(p)/(3\alpha(p)-3)}(Q)}+\Vert {\bf h}^1\Vert_{L^{5\alpha(p)/(4\alpha(p)-3)}(Q)}.$ 
The mapping $$A:\varphi\in Z\longmapsto (\varphi, \nabla\varphi)\in E$$ is an isometry from $Z$ into $E.$ Considering the graph $G=A(Z)\subset E$  which is a Banach space endowed with the norm  of $E$, the mapping 
$$B:(\varphi, \nabla\varphi)\in G\longmapsto \langle F,\varphi\rangle \in \mathbb{R}$$ is a linear and continuous functional on $G$ and $\Vert B\Vert_{G'}=\Vert F\Vert_{Z'}$. Then, by the Hahn-Banach Theorem, $B$ can be extended to a linear and continuous functional $\Phi:E\rightarrow \mathbb{R}$ with $\Phi|_{G}=B$ and $\Vert \Phi\Vert_{E'}=\Vert B\Vert_{G'}=\Vert F\Vert_{Z'}.$ By the Riesz representation Theorem, there exists $(g^0,{\bf g}^1)\in L^{5\alpha(p)/(2\alpha(p)+3)}(Q)\times (L^{5\alpha(p)/(\alpha(p)+3)}(Q))^3$ such that
$$\langle \Phi,(h^0,{\bf h}^1)\rangle=\int_0^T\int_\Omega g^0 h^0+\int_0^T\int_\Omega {\bf g}^1\cdot {\bf h}^1,\quad  \forall (h^0,{\bf h}^1)\in E.$$
On the other hand, using that $\Phi|_G=B$,  for all $\varphi\in Z,$ it holds
$$
\langle \Phi,(\varphi,\nabla\varphi)\rangle = \langle B,(\varphi,\nabla\varphi)\rangle
=\langle F,\varphi\rangle.$$
Then,  the proof can be deduced from the last two properties.
\end{proof}

%
 Once defined the  Banach space $Y_u$, we define the vector space
\begin{eqnarray}
&&X_u:=\left\{u\in X\,:\, \hbox{$u$ is a variational solution of a heat problem}  \right.
\nonumber\\
&&\hspace{1.5cm}\partial_tu-\Delta u=g^0-\nabla\cdot{\bf g}^1,\ u(0)=w_0,\ (-\nabla u + {\bf g}^1)\cdot {\bf n}\vert_{\partial \Omega}=0\ \mbox{ with }\nonumber \\
&&\hspace{1.5cm}
w_0\in L^{\alpha(p)}(\Omega), \left.g^0\in L^{5\alpha(p)/(2\alpha(p)+3)}(Q)\mbox{ and } {\bf g}^1\in L^{5\alpha(p)/(\alpha(p)+3)}(Q)^3\right\}, \label{xu}
\end{eqnarray}
endowed with the norm 
$$\Vert u\Vert_{X_u}=\Vert g^0-\nabla\cdot {\bf g}^1\Vert_{Z'}+\Vert w_0\Vert_{L^{\alpha(p)}}.
$$
 It holds that $(X_u,\Vert\cdot\Vert_{X_u})$ is a Banach space, where in particular $X_u \varsubsetneq X.$  
 
 Finally, we consider the Banach space
\begin{eqnarray}
&&X_v:=\{ v \in X_{5/2}: \, \partial_{\bf n} v \vert_{\partial \Omega}=0\}.\label{xv}
\end{eqnarray}

\

 Let  
$\mathcal{F}\subset L^{5/2}(L^{5/2+}(\Omega_c))$ be a nonempty, closed and convex set. 
 We assume the initial data $(u_0,v_0)\in L^{\alpha(p)}(\Omega)\times W^{6/5,5/2}(\Omega),
$ with $u_0\ge 0$ and $v_0\ge 0$ in $\Omega$, 
 and  $f\in\mathcal{F}$ the control which acts
on the  $v$-equation in a bilinear form. We define the admissible set  by 
\begin{equation}\label{s-ad}
\mathcal{S}_{ad}=\left\{
\begin{array}{lc}
(u,v,f)\in X_u\times X_v\times \mathcal{F}:\ (u,v)\  \mbox{is the strong solution of (\ref{eq1b})-(\ref{eq2})}\\
\ \ \ \ \ \ \ \ \ \ \ \ \ \ \ \ \ \ \ \ \ \ \ \ \ \ \ \ \ \ \ \ \ \ \mbox{with control}\ f
\end{array}
\right\}.
\end{equation}

%

Then, we will study the following (bilinear) optimal control problem:
\begin{equation}\label{C-3}
\left\{
\begin{array}{lc}
\min J(u,v,f),\\
\mbox{subject to}\ (u,v,f)\in \mathcal{S}_{ad},
\end{array}
\right.
\end{equation}
where $J: \left(L^{5p/2}(Q)\cap L^{10/3}(Q)\right)\times L^2(Q)\times L^{5/2}(L^{5/2+}(\Omega_c)) \to \mathbb{R}$ is the tracking cost functional defined by 
\begin{eqnarray*}\label{func}
J(u,v,f)&=&\gamma_u\left(\displaystyle\frac{1}{5p/2}\int_{0}^{T}\!\!\Vert u-u_{d}\Vert_{L^{5p/2}}^{5p/2} dt+\frac{1}{10/3}\int_0^T\|u-u_d\|^{10/3}_{L^{10/3}}dt\right)\\
&&+ 
\displaystyle\frac{\gamma_v}{2}\!\int_{0}^{T}\!\!\Vert v-v_{d}\Vert^{2} dt
+
\ \displaystyle\frac{\gamma_f}{5/2}\displaystyle\int_{0}^{T}\!\!\Vert f\Vert _{L^{5/2+}(\Omega_c)}^{5/2} dt.
\end{eqnarray*}
Here, the pair of functions $(u_d,v_d)\in (L^{5p/2}(Q)\cap L^{10/3}(Q))\times L^2(Q)$ represent  the desired states, and the real numbers $\gamma_u$, $\gamma_v$ and $\gamma_f$  measure the cost of the states and control, respectively. These coefficients satisfy
$$
\gamma_u>0,\ \gamma_v\ge 0, \mbox{ either } \gamma_f>0 \mbox{ or } \gamma_f=0 \mbox{ and } \mathcal{F} \mbox{ is bounded in } L^{5/2}(L^{5/2+}(\Omega_c)).
$$
The functional $J(u,v,f)$ describes the deviation of the cell density $u$ and the chemical concentration $v$ from some desired states $u_d$ and $v_d$ respectively, plus the cost of the control in the $L^{5/2}(L^{5/2+}(\Omega_c))$-norm. 
We are going to prove existence of (global) optimal solutions and analyze first order optimality conditions for (local) optimal solutions. 

\begin{definition}[Global optimal solution]\label{opt_sol}
A triple $(u^*,v^*,f^*)\in\mathcal{S}_{ad}$ is said a global optimal solution of problem \eqref{C-3} if
\begin{equation}\label{os-1}
J(u^*,v^*,f^*)=\min_{(u,v,f)\in\mathcal{S}_{ad}}J(u,v,f).
\end{equation}
\end{definition}
\begin{definition}[Local optimal solution]\label{loc_sol}
A triple $(u^*,v^*,f^*)\in\mathcal{S}_{ad}$ is said a local optimal solution of problem \eqref{C-3}, if there exists $\varepsilon>0$ such that  $J(u^*,v^*,f^*)\le J(u,v,f)$ for all $(u,v,f)\in\mathcal{S}_{ad}$ satisfying 
$
\|u-u^*\|_{X_u}+\|v-v^*\|_{X_v}+\|f-f^*\|_{L^{5/2}(L^{5/2+}(\Omega_c))}\le\varepsilon .
$
\end{definition}

\begin{theorem}[Existence of global optimum] \label{control}
Assume that 
$\mathcal{S}_{ad}\ne\emptyset.$ Then, there exists at least one global optimal solution of (\ref{C-3}).
\end{theorem}
\begin{remark}
Note that if $(u,v)$ is a weak solution of \eqref{eq1} or \eqref{eq1b} such that $J(u,v,f)<\infty$, then $f\in L^{5/2}(L^{5/2+}(\Omega_c)) $ and $u\in L^{5p/2}(Q)\cap L^{10/3}(Q)$. Then from Theorem \ref{strong}, the pair $(u,v)$ is the strong solution of  \eqref{eq1} or \eqref{eq1b} and $(u,v)\in X_u\times X_v$, hence
$
(u,v,f)\in \mathcal{S}_{ad}
$. 
In Remark~\ref{R4-1} we will give a particular case where  $\mathcal{S}_{ad}\ne\emptyset$ can be  ensured.
\end{remark}
The proof of Theorem \ref{control} is given in Subsection \ref{Sec:OptimalSolution}.
%

\


In order to identify the gradient of the reduced functional, we are going to consider the so-called equality operator 
$$\mathcal{S}:=(\mathcal{S}_1,\mathcal{S}_2):
{X}_{u}  \times {X}_{v} \times L^{5/2}(L^{5/2+}(\Omega_c))
\to Y_u \times Y_v ,$$ 
with $Y_u=Z'$ and $Y_v:=L^{5/2}(Q)$, 
defined as follows:
\begin{equation}\label{oper-s}
\left\{
\begin{array}{rcl}
\langle\mathcal{S}_1(u,v,f),\varphi\rangle_{Z'}
&:=&\displaystyle\int_0^T\langle\partial_t u, \varphi\rangle_{(H^1)'}
+\int_0^T\int_\Omega\nabla u\cdot\nabla \varphi 
+\int_0^T\int_\Omega u\nabla v\cdot\nabla \varphi
\\
&&-\displaystyle\int_0^T\int_\Omega(r\,u-\mu\, |u|^p) \varphi,\qquad \forall \,\varphi\in Z,
\\
\mathcal{S}_2(u,v,f)&:=&\partial_tv-\Delta v+v-|u|^p-f\,v\,1_{\Omega_c}
,\quad \hbox{in $L^{5/2}(Q)$}.
\end{array}
\right.
\end{equation}
Now, we are in position to identify the gradient of the reduced functional via the adjoint problem. 
\begin{theorem}
\label{FOC}
Let $(u^*,v^*,f^*)\in X_u\times X_v \times L^{5/2}(L^{5/2+}(\Omega_c))$ such that $\mathcal{S}(u^*,v^*,f^*)=\bf 0$. 
Then there exist neighborhoods $\mathcal{B}_{f^*}$ and $\mathcal{B}_{(u^{*},v^{*})},$ of $f^{*}$ and $(u^{*},v^{*})$ in the $L^{5/2}(L^{5/2+}(\Omega_c))$ and $X_u\times X_v$ topologies, respectively, such that the reduced cost functional $J_0:\mathcal{B}_{f^*}\rightarrow \mathbb{R}$ is well-defined as $J_0(f):=J(u_f,v_f,f)$, for any $f\in \mathcal{B}_{f^*}$, where $(u_f,v_f)$ is the unique strong solution of (\ref{eq1b})-(\ref{eq2}) in $\mathcal{B}_{(u^{*},v^{*})}$ with control $f$. 
 Moreover, $J_0$ is Fr\'echet differentiable, and its derivative is 
\begin{eqnarray}\label{cop2a}
J_0'(f^*)[f]=
\int_0^T\int_{\Omega_c}\left(
\gamma_f 
\frac{{\rm sgn}(f^{*})\vert f^{*}\vert^{3/2+}}
{\| f^* \|_{L^{5/2+}(\Omega_c)}^{0+}}
+{v^*}\eta\right)f ,\quad  \forall \, {f}\in L^{5/2}(L^{5/2+}(\Omega_c)),
\end{eqnarray}
where $\eta$ corresponds to the second component of the pair $(\sigma,\eta)\in  Z\times L^{5/3}(Q)$ satisfying the following adjoint system  (in the corresponding weak sense)
\begin{equation}\label{adjz0}
\left\{
\begin{array}{l}%
  -\partial_t\sigma-\Delta\sigma+\nabla\sigma\cdot\nabla{v^*}-p{(u^*)}^{p-1}\eta-r\sigma+p\mu (u^*)^{p-1}\sigma=
  \gamma_u  h({u^*}-u_d) ,\\ 
  -\partial_t\eta-\Delta\eta-\nabla\cdot({u^*}\nabla\sigma)+\eta-{f^*}\eta\,1_{\Omega_c}=\gamma_v({v^*}-v_d),\\
  \sigma(T)=0,\ \eta(T)=0\ \mbox{ in }\ \Omega,\\
{\partial_{\bf n}\sigma}=0,\ \ {\partial_{\bf n}\eta}=0\ \mbox{ on }\  (0,T)\times\partial\Omega,
 \end{array}
\right.
\end{equation}
with $h(w):=|w|^{(5p-4)/2} w+|w|^{4/3}w$.
 \end{theorem}

 The proof of Theorem \ref{FOC} is given in Subsection \ref{Sec:FOC}.
  \begin{remark}
 For $p\le 4/3$, the regularity of the solution $(\sigma,\eta)$ of the adjoint system \eqref{adjz0} can be improved to  
 $(\sigma,\eta)\in W_2 \times L^{2}(Q)$  (see Proposition \ref{Reg-mult} below).
%
 \end{remark}
  \begin{corollary}[First-order necessary optimality conditions]  \label{optim-cond}
 Let $(u^*,v^*,f^*)\in \mathcal{S}_{ad}$ be a local optimal solution of  \eqref{C-3}.
 Then the following optimality condition holds
\begin{equation}\label{L-3}
\int_0^T\int_{\Omega_c}\left(
\gamma_f 
\frac{{\rm sgn}(f^{*})\vert f^{*}\vert^{3/2+}}
{\| f^* \|_{L^{5/2+}(\Omega_c)}^{0+}}
+{v^*}\eta\right) (f-f^*)\ge0,\ \ \ \forall f\in\mathcal{F},
\end{equation}
where $(\sigma,\eta)$ is the solution of the adjoint system \eqref{adjz0}.
 \end{corollary}
 The proof of Corallary \ref{optim-cond} is given in Subsection \ref{Sec:optim-cond}.
 
\section{Existence of weak and strong solutions}

In this section we will prove the existence of weak  solutions of systems \eqref{eq1} and \eqref{eq1b} (Theorem \ref{teo1}), 
as well as the proof of existence and uniqueness of strong solution under the regularity criterion (Theorem \ref{strong}).
First of all, we introduce some formal computations to give the main ideas to treat the existence of weak solutions. Then, to justify these formal computations, the weak solution is obtained as the limit of a sequence of strong solutions of a family of regularized problems. 
The regularization incorpores an elliptic problem for the variable $v$ with an artificial diffusion with parameter $\varepsilon>0.$ 
Finally, we prove that any  weak solution satisfying  the regularity criterion \eqref{reg-crit} has global-in-time strong regularity.

\subsection {Formal computations related to weak solutions.}\label{Sec:FormalComput}

An important step to prove the existence of weak solutions  for  (\ref{eq1})-(\ref{eq2}) or \eqref{eq1b}--(\ref{eq2})  is to deduce an energy inequality depending on the norm $L^\infty(L^2)\cap L^2(H^1)$ of $u^{p/2}$ and $\nabla v$. 

Starting with the non-logistic problem (\ref{eq1}), 
multiplying (\ref{eq1})$_1$ by $\frac{1}{p-1}u^{p-1}$ and \eqref{eq1}$_2$ by $-\frac1p\Delta v$; then, integrating by parts in the spatial variable, adding the corresponding results in order to cancel the chemotaxis and production effects, and then using the H\"older, Young and Gagliardo-Nirenberg inequalities (see details in Subsection \ref{subsec:regularized}) we arrive at:
\begin{eqnarray} \label{B-20}
\frac{d}{dt}\!\left(\frac{1}{p(p-1)}\|u^{p/2}\|^2\!+\!\frac{1}{2p}\|\nabla v\|^2
\right)\!+\!\frac{4}{p^2}\|\nabla(u^{p/2})\|^2\!+\!\frac{1}{4p}\|\Delta v\|^2\!+\!\frac1p\|\nabla v\|^2\le\! {C\|f\|^{5/2}_{L^{5/2}}\|v\|^2_{H^1}.}
\end{eqnarray}

Hereafter, $C>0$ denotes different constants. On the other hand, integrating  (\ref{eq1})$_2$ in $\Omega$, then multiplying the resulting equation by $\frac1p\int_\Omega v$ and applying the H\"older and Young inequalities, we can obtain
\begin{eqnarray}
\frac{1}{2p}\frac{d}{dt}\left(\int_\Omega v\right)^2+\frac{1}{2p}\left(\int_\Omega v\right)^2
&\le&\frac{1}{p}\left(\int_\Omega u^p\right)^2+\frac{1}{p}C\|f\|^2\|v\|^2\nonumber\\
&\le&\frac{1}{p}\|u\|^{2p}_{L^p}+\frac{1}{p}C\|f\|^2\|v\|^2.\label{B-30}
\end{eqnarray}
%
%
In order to bound the term $\|u\|^{2p}_{L^p}$,  
 from the Gagliardo-Nirenberg interpolation inequality of $L^2(\Omega)$ between $H^1(\Omega)$ and $L^{2/p}(\Omega)$, we have
\begin{equation}\label{B-4z}
\|u\|^{2p}_{L^p}=\|u^{p/2}\|^4\le C\left(\|\nabla (u^{p/2})\|^{12(p-1)/(3p-1)}\|u^{p/2}\|^{8/(3p-1)}_{L^{2/p}}+\|u^{p/2}\|^4_{L^{2/p}}\right).
\end{equation}
At this point, we distinguish three cases, namely, $p\in(1,5/3),$ $p=5/3,$ and $p\in(5/3,\infty).$ For $p\in(1,5/3]$ we can obtain energy estimates only with the diffusion, while in the last case $p\in(5/3,\infty)$ we will need the logistic term. More exactly, we have:
\begin{enumerate}
\item 
\underline {$p\in\left(1,5/3\right)$ and non-logistic case.} Then $12(p-1)/(3p-1)<2;$ thus we can apply the Young inequality on the right side of \eqref{B-4z}, and take into account that $\|u(t,\cdot )\|_{L^1}=\|u_0\|_{L^1}$ 
 (which can be  deduced integrating (\ref{eq1})$_1$ in $\Omega$), to get
\begin{eqnarray}
\frac{1}{p}\|u\|^{2p}_{L^p}
&\le&\frac{2}{p^2}\|\nabla (u^{p/2})\|^2+C\|u^{p/2}\|^{8/(5-3p)}_{L^{2/p}}
+C\|u^{p/2}\|^4_{L^{2/p}}\nonumber\\
&=&\frac{2}{p^2}\|\nabla (u^{p/2})\|^2+C\left(\int_\Omega u_{0}\right)^{8/(5-3p)}+C\left(\int_\Omega u_{0}\right)^4.\label{B-50}
\end{eqnarray}
Consequently, adding \eqref{B-20} and \eqref{B-30}, using \eqref{B-50}, we conclude the energy inequality
\begin{eqnarray}\label{energy0}
\frac{d}{dt}\left( \frac{1}{p(p-1)}\Vert  u^{p/2}\Vert^2
+\frac{1}{2p}\Vert v\Vert_{H^1}^2\right)
&+&{\frac{C}{p}\Vert v\Vert_{H^2}^2}+\frac{2}{p^2}\Vert \nabla (u^{p/2})\Vert^2
\nonumber\\
&\leq & {C\left(\|f\|^{5/2}_{L^{5/2}}+\Vert f\Vert^2_{L^{5/2}}\right)\Vert v\Vert^2_{H^1}+C}.
\end{eqnarray}
\item
\underline{ $p=5/3$ and non-logistic case.} Then, from (\ref{B-4z}) we get
\begin{eqnarray}
\frac{1}{p}\|u\|^{2p}_{L^p}&=&\frac{3}{5}\|u^{5/6}\|^4\le\frac{3}{5}\left(C\|\nabla (u^{5/6})\|^{2}\|u^{5/6}\|^{2}_{L^{6/5}}+C\|u^{5/6}\|^4_{L^{6/5}}\right)\nonumber\\
&\le& C\left(\|\nabla (u^{5/6})\|^{2}\left(\int_\Omega u_{0}\right)^{5/3}+\left(\int_\Omega u_{0}\right)^{10/3}\right)
\le C\left(\|\nabla (u^{5/6})\|^{2}+1 \right).
\label{B-4zz}
\end{eqnarray}
Then, a linear combination between (\ref{B-20}) and (\ref{B-30}) (and using (\ref{B-4zz})) in order to absorb the term 
$\|\nabla u^{5/6}\|^{2}$, also implies the energy inequality (\ref{energy0}).

\item \underline{$p>1$ and logistic case.} Here we are only able to bound the energy considering the logistic problem (\ref{eq1b}). In this case, integrating in $\Omega$ the
$u$-equation, one has the estimates
$$
\Vert u(t,\cdot) \Vert_{L^1} \le \max\{\Vert u_0 \Vert_{L^1},\left(r/\mu\right)^{{1}/{p-1}} \vert \Omega \vert \}
\quad \forall t>0,$$
$$\Vert u\Vert_{L^p(Q)}\leq C.$$
Thus the term $\|u\|_{L^p}^{2p}$ can be controlled by  interpolation 
 as follows:
\begin{equation}\label{ps-110}
\|u\|_{L^p}^{2p}=\|u^{p/2}\|^4
\le C\|u^{p/2}\|_{L^{2(2p-1)/p}}^{2(2p-1)/p}\|u^{p/2}\|_{L^{2/p}}^{2/p}
= C(\|u_0\|_{L^{1}},r,\mu)\, \|u\|_{L^{2p-1}}^{2p-1} \, .
\end{equation}

\medskip

Consequently, using an adequate linear combination of (\ref{B-20}) and (\ref{B-30}) (taking into account (\ref{ps-110})), we can derive (see details in Subsection \ref{subsec:regularized})
\begin{eqnarray}
&&\frac{d}{dt}\left(\frac{1}{p(p-1)}\|u^{p/2}\|^2+\frac{1}{2p}\|v\|^2_{H^1}\right)
+\frac{4}{p^2}\|\nabla(u^{p/2})\|^2
+\frac Cp\|v\|^2_{H^2}
+\frac{\mu}{p-1}\|u\|_{L^{2p-1}}^{2p-1}\nonumber\\
&&\ \ \le \frac r{p-1}\|u\|^p_{L^p}+C\left(\|f\|^{5/2}_{L^{5/2}}+\Vert f\Vert^2_{L^{5/2}}\right)\|v\|^2_{H^1}+C.\label{ps-120} 
\end{eqnarray}
\end{enumerate}
Energy inequalities (\ref{energy0}) and (\ref{ps-120}) are the keys to prove the existence of global weak solutions.

\subsection{A regularized system}\label{subsec:regularized}
In order to prove Theorem \ref{teo1}, we will study the following family of regularized problems related to systems \eqref{eq1}-\eqref{eq2}.  Given $\varepsilon>0,$ find $({u}_\varepsilon,w_\varepsilon)$ nonnegative functions  such that: {\begin{equation}\label{eq1_aprox}
\left\{
\begin{array}{rcl}
\partial_tu_\varepsilon-\Delta u_\varepsilon
&=&\nabla\cdot(u_\varepsilon\nabla {v(w_\varepsilon}))
 +r\, u_\varepsilon-\mu \, u_\varepsilon^p
\ \  \mbox{in $Q$},\\
\partial_tw_\varepsilon-\Delta w_\varepsilon+w_\varepsilon
&=&u_\varepsilon^p+{f\, v(w_\varepsilon)_+\, 1_{\Omega_c}}\ \  \mbox{in $Q$}, \\
\left(u_\varepsilon(0),w_\varepsilon(0)\right)&=&\left(u_{0,\varepsilon},w_{0,\varepsilon}\right)\ \  \mbox{in $\Omega$},\\
 \partial_{\bf n} u_\varepsilon &=&\partial_{\bf n} w_\varepsilon=0\ \  \mbox{on $(0,T)\times\partial\Omega$},
\end{array}
\right.
\end{equation}
where $v(w_\varepsilon(t,\cdot))= v_\varepsilon(t,\cdot)$  is the unique solution of the elliptic problem
\begin{equation}\label{elip}
\left\{
\begin{array}{lc}
v_\varepsilon-\varepsilon\Delta v_\varepsilon=w_\varepsilon\quad \mbox{in}\ \Omega,\\
\partial_{\bf n}v_\varepsilon=0\quad \mbox{on}\ \partial\Omega.
\end{array}
\right.
\end{equation}
In (\ref{eq1_aprox}), $(u_{0,\varepsilon},w_{0,\varepsilon})$ are regular initial data satisfying 
\begin{equation}\label{elip1}
\left\{
\begin{array}{lc}
u_{0,\varepsilon}\in L^{{\infty}-}(\Omega),\ u_{0,\varepsilon}\ge 0,
\\
u_{0,\varepsilon}\rightarrow u_0\ \mbox{in}\ L^p(\Omega),\ \mbox{as}\ \varepsilon\rightarrow 0,\\
v_{0,\varepsilon}\in W^{2+6/5,5/2}_{\bf n}(\Omega),\ 
v_{0,\varepsilon}\rightarrow v_0\ \mbox{in}\ H^{1}(\Omega),\ \mbox{as}\ \varepsilon\rightarrow 0,
\\ 
w_{0,\varepsilon}=v_{0,\varepsilon}-\varepsilon\Delta v_{0,\varepsilon}\in {W}^{6/5,5/2}(\Omega).
\end{array}
\right.
\end{equation}

\begin{proposition}[Strong solutions of  \eqref{eq1_aprox}-\eqref{elip1}]\label{solution_regu}
Let $u_{0,\varepsilon},v_{0,\varepsilon},w_{0,\varepsilon}$ satisfying  \eqref{elip1} and $f\in L^{5/2}(Q_c)$. There exists a unique solution $(u_\varepsilon,w_\varepsilon)\in A\times X_{5/2}$ of regularized problem \eqref{eq1_aprox} in $(0,T)$ with $u_\varepsilon\ge0$ a.e.\ in $Q,$ where $A:=\{ u\in L^{\infty-}(Q):\ \nabla u\in L^2(Q),\ \partial_t u\in L^2((H^1)')\}$.
\end{proposition} 
\begin{proof}
We will use the Leray-Schauder fixed-point theorem. We recall the space $W_2=L ^\infty(L^2)\cap L^2(H^1),$ and consider the operator
$$
\mathcal{L}:L^{\infty-}(Q)\times W_2\to A\times X_{5/2}\hookrightarrow L^{\infty-}(Q)\times W_2,
$$
defined by $\mathcal{L}(\overline{u}_\varepsilon,\overline{w}_\varepsilon)=(u_\varepsilon,w_\varepsilon)$, where $(u_\varepsilon,w_\varepsilon)$ is the solution of the linear (decoupled) problem 
\begin{equation}\label{A1}
\left\{
\begin{array}{l}
\partial_tu_\varepsilon-\Delta u_\varepsilon
-\nabla\cdot({u}_{\varepsilon+}\nabla v(\overline{w}_\varepsilon))
-r\, u_{\varepsilon}+\mu \,  (\overline u_{\varepsilon +})^{p-1}{u_{\varepsilon}}=0
\ \mbox{ in }\ Q,\\
\partial_tw_\varepsilon-\Delta w_\varepsilon+w_\varepsilon
=  (\overline u_{\varepsilon +} )^{p-1}{u_{\varepsilon}}
+f v(\overline{w}_{\varepsilon })_+\,1_{\Omega_c}\ \mbox{ in }\ Q,\\
u_\varepsilon(0)=u_{0, \varepsilon},\ w_\varepsilon(0)=
w_{0,\varepsilon}\ \mbox{ in }\ \Omega,\\
{\partial_{\bf n} u_\varepsilon}=0,\ {\partial_{\bf n} w_\varepsilon}=0\ \mbox{ on }\ (0,T)\times\partial\Omega,
\end{array}
\right.
\end{equation}
where $v(\overline{w}_\varepsilon)$ 
is the unique solution  of problem \eqref{elip} (changing $w_\varepsilon$ by $\overline w_\varepsilon$),
 and $\overline{u}_{\varepsilon+}$ 
denotes the positive part of $\overline{u}_\varepsilon$.
 
\

{\it  \underline{Step 1:} The operator $\mathcal{L}$ is well-defined, and  $A\times X_{5/2}$ is compactly embedded in $L^{\infty-}(Q)\times W_2$.}
\vspace{0.1cm}

Let $(\overline{u}_\varepsilon,\overline{w}_\varepsilon)\in L^{\infty-}(Q)\times W_2$; 
then, from the $H^2$ and $H^3$-regularity of linear problem 
(\ref{elip})
we deduce that $\overline{v}_\varepsilon:=v(\overline{w}_\varepsilon) \in L^\infty(H^2)\cap L^2(H^3)$. 
Thus, we have that $\nabla \overline{v}_\varepsilon\in L^\infty(H^1)\cap L^2(H^2)\hookrightarrow L^{10}(Q);$
hence, using that $\overline{u}_\varepsilon\in L^{\infty-}(Q),$ from Lemma \ref{lem-B}, for any $2\le q<\infty$, there exists
a unique weak solution $u_\varepsilon\in A$ of problem \eqref{A1}$_1$ and satisfies  the following estimate:
\begin{equation}\label{A-1-1}
\|u_\varepsilon\|_{A}\le C\left(\|u_{0,\varepsilon}\|_{L^{\infty-}},\|\overline{u}_\varepsilon\|_{L^{\infty-}(Q)},\|\overline{w}_\varepsilon\|_{W_2}\right).
\end{equation}
Moreover, $u_\varepsilon\ge 0$ a.e.\ in $Q$. Indeed, by testing \eqref{A1}$_1$ by $u_{\varepsilon-}:=\min\{u_\varepsilon,0\}\in  L^{\infty-}(Q)$, 
and taking into account that $u_{\varepsilon-}=0$ if $u_\varepsilon\ge0$, $\nabla u_{\varepsilon-}=\nabla u_\varepsilon$ if $u_\varepsilon\le 0$, and $\nabla u_{\varepsilon-}=0$ if $u_\varepsilon>0$, we have
$$
\frac12\frac{d}{dt}\|u_{\varepsilon-}\|^2+\|\nabla u_{\varepsilon-}\|^2
+\mu\Vert (\overline u_{\varepsilon+})^{p-1}u^2_{\varepsilon-}\Vert_{L^1}
\le \, r\, \|u_{\varepsilon-}\|^2,
$$
which together with $u_{0,\varepsilon}\ge 0$ a.e. in $\Omega$, implies that $u_{\varepsilon-}\equiv0$; consequently $u_\varepsilon\ge0$ a.e. in $Q$.

Now, using that $v(\overline{w}_\varepsilon)\in L^\infty(H^2)\subset L^\infty(Q),$ and $u_\varepsilon\in A\subset L^{\infty-}$ we have $(\overline u_{\varepsilon +} )^{p-1}{u_{\varepsilon}}+f {v(\overline{w}_\varepsilon})_+\,1_{\Omega_c}\in L^{5/2}(Q)$. 
Then, from parabolic regularity there exists a unique solution $w_\varepsilon\in X_{5/2}$ of \eqref{A1}$_2$ such that
\begin{equation}\label{A-2}
\|w_\varepsilon\|_{X_{5/2}}\le C\left(\|w_{0,\varepsilon}\|_{W^{6/5,5/2}},
\|\overline{u}_\varepsilon\|_{L^{\infty-}(Q)},
\|{u}_\varepsilon\|_{A},\|\overline{w}_\varepsilon\|_{W_2},\|f\|_{L^{5/2}(Q_c)}\right).
\end{equation}

Therefore, the operator $\mathcal{L}$ is well-defined from $L^{\infty-}(Q)\times W_2$ in $A\times X_{5/2}$, and from \eqref{A-1-1}-\eqref{A-2}, it maps  bounded sets of $L^{\infty-}(Q)\times W_2$ into bounded sets of $A\times X_{5/2}$. On the other hand, the product space $A\times X_{5/2}$ is compactly embedded in $L^{\infty-}(Q)\times W_2$ as a direct consequence of \cite[Corollary 4]{Simon} and \cite[Th\'eor\`eme 5.1]{Lions}.
\vspace{0.1cm}

{\it \underline{Step 2:} The possible fixed points $(u_\varepsilon,w_\varepsilon)$ of $\lambda \mathcal{L}$ are bounded in $L^{\infty-}(Q)\times W_2$, independently of the parameter $\lambda\in[0,1]$.
}

\vspace{0.1cm}

Let $\lambda\in(0,1]$ (the case $\lambda=0$ is clear). We observe that if $(u_\varepsilon,w_\varepsilon)$ is a fixed point of $\lambda \mathcal{L}$, then the pair $(u_\varepsilon,w_\varepsilon)$ belongs to $A\times X_{5/2}$, $u_\varepsilon\ge0$,  and the following problem holds  a.e. in $Q$: 
\begin{equation}\label{B}
\left\{
\begin{array}{l}
\partial_tu_\varepsilon-\Delta u_\varepsilon
-\nabla\cdot({u}_{\varepsilon+}\nabla v_\varepsilon)
-r\,  u_{\varepsilon} +\,\mu \,   u_{\varepsilon }^{p}=0
\ \mbox{ in }\ Q,\\
\partial_tw_\varepsilon-\Delta w_\varepsilon+w_\varepsilon
=  u_{\varepsilon}^{p}
+\lambda f\, {v}(w_{\varepsilon})_+\,1_{\Omega_c}\ \mbox{ in }\ Q,\\
u_\varepsilon(0)
= u_{0,\varepsilon},\ w_\varepsilon(0)
=w_{0,\varepsilon}\ \mbox{ in }\ \Omega,\\
{\partial_{\bf n} u_\varepsilon}=0,\ {\partial_{\bf n} w_\varepsilon}=0\ \mbox{ on }\  (0,T)\times\partial\Omega.
\end{array}
\right.
\end{equation}

On the other hand, integrating in $\Omega$ equation \eqref{B}$_1$, 
we obtain
\begin{equation}\label{B-1}
\frac{d}{dt} \|u_{\varepsilon}(t)\|_{L^1}
+  \mu\, \|u_{\varepsilon}(t)\|_{L^p}^p 
=
r\, \|u_{\varepsilon}(t)\|_{L^1} 
\ \ \forall t>0.
\end{equation}
By bounding from below $\|u_{\varepsilon}(t)\|_{L^p}^p \ge |\Omega| ^{p-1}
\|u_{\varepsilon}(t)\|_{L^1}^p$ in \eqref{B-1}, it can be deduced that $u_\varepsilon$ is bounded in $L^\infty(L^1)$ by a standard comparison argument. In fact, one has 
\begin{equation}\label{L1-u}
\Vert u_\varepsilon(t,\cdot) \Vert_{L^1} \le K_0, 
\quad \hbox{where} \quad 
K_0=
\left\{
\begin{array}{ll}
  \Vert u_{0,\varepsilon} \Vert_{L^1} &  \hbox{if $r=\mu=0$,}    \\
  \max\{\Vert u_{0,\varepsilon} \Vert_{L^1},\left(r/\mu\right)^{{1}/{(p-1)}} \vert \Omega \vert \} &    \hbox{if $r,\mu>0$.} 
\end{array}
\right.
\end{equation}

Now, we can prove rigorously  the energy estimates given (formally) in Subsection \ref{Sec:FormalComput}. Indeed, by adding \eqref{B}$_1$ by $\frac{1}{p-1}u_\varepsilon^{p-1}$ and \eqref{B}$_2$ (rewritten in terms of $ v_\varepsilon$) by $-\frac1p\Delta v_\varepsilon,$ then chemotaxis and production terms cancel; 
integrating by parts, using the 3D interpolation inequality $\|v\|_{L^{10}}\le C\|v\|^{4/5}_{H^1}\|v\|^{1/5}_{H^2}$
and  the H\"older and Young inequalities  
we can obtain 
\begin{eqnarray}
&&\frac{d}{dt}\left(\frac{1}{p(p-1)}\|u_\varepsilon^{p/2}\|^2+\frac{1}{2p}\|\nabla v_\varepsilon\|^2
+\frac{\varepsilon}{2p}\|\Delta v_\varepsilon\|^2\right)+\frac{4}{p^2}\|\nabla(u_\varepsilon^{p/2})\|^2+\frac1p\|\Delta v_\varepsilon\|^2\nonumber\\
&&\ \ +\frac1p\|\nabla v_\varepsilon\|^2+ \frac{\varepsilon}{p}\|\Delta v_\varepsilon\|^2+\frac{\varepsilon}{p}\|\nabla(\Delta v_\varepsilon)\|^2
+ \frac{\mu }{p-1} \| u_\varepsilon \|_{L^{2p-1}}^{2p-1}
\nonumber\\
&&
\le {\frac{\lambda}{p}\|f\|_{L^{5/2}}\|v_\varepsilon\|_{L^{10}}\|\Delta v_\varepsilon\|}
+ \frac{ r }{p-1} \| u_\varepsilon \|_{L^{p}}^{p}
\nonumber\\
&&\le {\frac{C}{p}\|f\|_{L^{5/2}}\|v_\varepsilon\|^{4/5}_{H^{1}}\|v_\varepsilon\|^{1/5}_{H^{2}}\|\Delta v_\varepsilon\|}
+ \frac{ r }{p-1} \| u_\varepsilon ^{p/2} \|^{2}
 \nonumber\\
&&\le C_\delta \|f\|^{5/2}_{L^{5/2}}\|v_\varepsilon\|^2_{H^1}+ 
\delta\, \|v_\varepsilon\|^2_{H^2}
+ \frac{ r }{p-1} \| u_\varepsilon ^{p/2} \|^{2} , \label{B-2-1}
\end{eqnarray}
for any $\delta>0.$ 
Also, integrating \eqref{B}$_2$ in $\Omega$, then multiplying the resulting equation by $\frac1p\int_\Omega v_\varepsilon$ and applying the H\"older and Young inequalities, we can obtain
\begin{eqnarray}
\frac{1}{2p}\frac{d}{dt}\left(\int_\Omega v_\varepsilon\right)^2+\frac{1}{2p}\left(\int_\Omega v_\varepsilon\right)^2
&\le&  \frac{\delta}{p}\left(\int_\Omega u_\varepsilon^p\right)^2
+C_\delta \frac{\lambda^2}{p} \|f\|^2_{L^{5/2}} \|v_\varepsilon\|^2_{L^{5/3}}
\nonumber\\
&\le&\frac{\delta}{p}\|u_\varepsilon\|^{2p}_{L^p}
+\frac {C_\delta} p \|f\|^2_{L^{5/2}}\|v_\varepsilon\|^2_{H^1}.\label{B-3}
\end{eqnarray}
At this point, we split into three  cases $p\in (1,5/3)$ and  $p=5/3$ in the non-logistic problem,  and $p>1$ in the logistic case.

\

{\it \underline{Case $p\in (1,5/3)$ and non-logistic problem.}}

From the Gagliardo-Nirenberg interpolation inequality  we have
\begin{equation}\label{B-4}
\begin{array}{rcl}
\|u_\varepsilon\|^{2p}_{L^p}=\|u_\varepsilon^{p/2}\|^4
&\le& C
\left(\|\nabla u_\varepsilon^{p/2}\|^{12(p-1)/(3p-1)}\|u_\varepsilon^{p/2}\|^{8/(3p-1)}_{L^{2/p}}+\|u_\varepsilon^{p/2}\|^4_{L^{2/p}}\right)\\
& \le &
C \left(\|\nabla u_\varepsilon^{p/2}\|^{12(p-1)/(3p-1)} K_0^{4p/(3p-1)} + K_0^{2p}\right).
\end{array}
\end{equation}

Since  $p\in (1,5/3)$, then $12(p-1)/(3p-1)<2;$ 
thus, we can apply the Young inequality on the right-hand side of \eqref{B-4}, and use \eqref{L1-u} in order to obtain 
\begin{eqnarray}
\frac{1}{p}\|u_\varepsilon\|^{2p}_{L^p}
&\le&\frac{2}{p^2}\|\nabla u_\varepsilon^{p/2}\|^2+C\|u_\varepsilon^{p/2}\|^{8/(5-3p)}_{L^{2/p}}
+C\|u_\varepsilon^{p/2}\|^4_{L^{2/p}}\nonumber\\
&=&\frac{2}{p^2}\|\nabla u_\varepsilon^{p/2}\|^2
+C\, K_0^{8/(5-3p)}
+C\, K_0^4.\label{B-5}
\end{eqnarray}
Consequently, from \eqref{B-3} (taking $\delta=1$) and \eqref{B-5} we arrive at
\begin{equation} \label{B-6}
\frac{1}{2p}\frac{d}{dt}\left(\int_\Omega v_\varepsilon\right)^2+\frac{1}{2p}\left(\int_\Omega v_\varepsilon\right)^2
\le\frac{2}{p^2}\|\nabla u_\varepsilon^{p/2}\|^2
+C\, K_0^{8/(5-3p)}
+C\, K_0^4
+ C\, \|f\|^2_{L^{5/2}}\|v_\varepsilon\|^2_{H^1}.
\end{equation}
Hence, adding \eqref{B-2-1} and \eqref{B-6}, and  considering the equivalent norms given in (\ref{norm-1}) and (\ref{norm-2}),  we can obtain 
\begin{eqnarray}
&&\frac{d}{dt}\left(\frac{1}{p(p-1)}\|u_\varepsilon^{p/2}\|^2
+{\frac{1}{2p}\|v_\varepsilon\|^2_{H^1}}
+\frac{\varepsilon}{2p}\|\Delta v_\varepsilon\|^2\right)
+\frac{2}{p^2}\|\nabla(u_\varepsilon^{p/2})\|^2
+\frac{1}{2p}\|v_\varepsilon\|^2_{H^2}
\nonumber\\
&&\ \ 
+\frac{\varepsilon}{p}\|\Delta v_\varepsilon\|_{H^1}^2
\le
C\, K_0^{8/(5-3p)}
+C\, K_0^4
 +{C\left(\|f\|^{5/2}_{L^{5/2}}+\|f\|^2_{L^{5/2}}\right)\|v_\varepsilon\|^2_{H^1}}.\label{B-7}
\end{eqnarray}

Thus, from \eqref{B-7} and the Gronwall lemma we deduce that there exists a positive constant $K_\varepsilon:=K_\varepsilon(T,\|u_{0,\varepsilon}^{p/2}\|,\|v_{0,\varepsilon}\|_{H^2},{\|f\|_{L^{5/2}(Q)}})$ such that
$
\|u_\varepsilon^{p/2}\|_{W_2}\le K_\varepsilon
$
and 
$
\|v_\varepsilon\|_{L^\infty(H^2)\cap L^2(H^3)}\le K_\varepsilon.
$
In particular, $w_\varepsilon$ is bounded in $W_2$ (with respect to $\lambda$).

\

{\it \underline{Case $p=5/3$ and non-logistic problem.}}
 
From (\ref{B-4}) we get
\begin{eqnarray}
\|u_\varepsilon\|^{10/3}_{L^{5/3}}
&=&
\|u_\varepsilon^{5/6}\|^4
\le
C\left(\|\nabla (u_\varepsilon^{5/6})\|^{2}\|u_\varepsilon^{5/6}\|^{2}_{L^{6/5}}
+\|u_\varepsilon^{5/6}\|^4_{L^{6/5}}\right)\nonumber\\
&\le& C\left(\|\nabla (u_\varepsilon^{5/6})\|^{2}K_0^{5/3}+K_0^{10/3}\right).\label{B-4zz3}
\end{eqnarray}
Then, by adding (\ref{B-2-1}) and (\ref{B-3})  (taking $\delta$ small enough), and using (\ref{B-4zz3}), we deduce an energy inequality similar to (\ref{B-7}).

\

{\it \underline{Case $p>1$ and logistic problem}.} 

 In this case, 
the diffusion is not enough to control the norm $\|u_\varepsilon\|_{L^p}^{2p}$. In order to overcome this difficulty, we use $L^s$-interpolation taking into account the presence of the $L^{2p-1}$-norm coming from the logistic term. In fact, interpolating the $L^{2p}$-norm between in $L^1(\Omega)$ and $L^{2p-1}(\Omega)$, and using the Young inequality, we arrive at
\begin{equation}
 \|u_\varepsilon\|_{L^p}^{2p}
\le \|u_\varepsilon\|_{L^{2p-1}}^{2p-1}\|u_{0, \varepsilon}\|_{L^{1}}
\le C\|u_\varepsilon\|_{L^{2p-1}}^{2p-1}.\label{Bb-4}
\end{equation}
Consequently, by adding  (\ref{B-2-1}) and (\ref{B-3})  (taking $\delta$ small enough) 
considering the equivalent norms given in (\ref{norm-1}) and (\ref{norm-2}), 
and using estimate $(\ref{Bb-4})$, 
we can absorbe the term $\|u_\varepsilon\|_{L^{2p-1}}^{2p-1}$ and arrive at
\begin{eqnarray}
&&\frac{d}{dt}\left(\frac{1}{p(p-1)}\|u_\varepsilon^{p/2}\|^2
+\frac{1}{2p} 
\| v_\varepsilon\|_{H^1}^2
+\frac{\varepsilon}{2p}\|\Delta v_\varepsilon\|^2\right)\nonumber\\
&&\ \ +\frac{4}{p^2}\|\nabla(u_\varepsilon^{p/2})\|^2
+ \left( \frac\mu{p-1}-\frac\delta p \widetilde C \right)\|u_\varepsilon\|_{L^{2p-1}}^{2p-1}
+{\frac{C}{2p}\|v_\varepsilon\|^2_{H^2}}
+\frac{\varepsilon}{p}\|\Delta v_\varepsilon\|_{H^1}^2\nonumber\\
&&\ \ \le r\|u_\varepsilon\|^p_{L^p}
+{C\left(\|f\|^{5/2}_{L^{5/2}}+\|f\|^2_{L^{5/2}}\right)\|v_\varepsilon\|^2_{H^1}}+C.
\label{Bb-5} 
\end{eqnarray}
Then, 
from the Gronwall lemma,  there exists 
$\widetilde{K}_\varepsilon:=\widetilde{K}_\varepsilon\left(T,\|u_{0,\varepsilon}\|_{L^p},\|v_{0,\varepsilon}\|_{H^2},\|f\|_{L^{5/2}(Q_c)}\right)>0$ such that
$
\|v_\varepsilon\|_{L^\infty(H^2)\cap L^2(H^3)}\le \widetilde{K}_\varepsilon.
$
In particular, $w_\varepsilon$ is bounded in $W_2$ (with respect to $\lambda$).

\

Now, we are going to get the bound for $u_\varepsilon$ in $L^{\infty-}(Q)$.  Indeed, by testing \eqref{B}$_1$ by $u^{q-1}_\varepsilon,$ $2\leq q<\infty$, considering the 3D interpolation inequality $\|u\|_{L^3}\le C\|u\|^{1/2}\|u\|^{1/2}_{H^1}$, and applying the H\"older and Young inequalities,  we have
\begin{eqnarray*}
&&\!\frac1q\frac{d}{dt}\|u_\varepsilon^{q/2}\|^2
\!+\!\frac{4(q-1)}{q^2}\|\nabla (u_\varepsilon^{q/2})\|^2
\!+\!\mu\|u_\varepsilon\|^{p+q-1}_{L^{p+q-1}}\!=\!\frac{-2(q-1)}{q}\!\int_\Omega u_\varepsilon^{q/2}\nabla v_\varepsilon\!\cdot\! \nabla(u_\varepsilon^{q/2})\!+\!r\Vert u_\varepsilon^{q/2}\Vert^2\nonumber\\
&&\ \  \leq \frac{2(q-1)}{q}\Vert u_\varepsilon^{q/2}\Vert_{L^3}\Vert \nabla v_\varepsilon\Vert_{L^6}\Vert \nabla (u_\varepsilon^{q/2})\Vert+\!r\Vert u_\varepsilon^{q/2}\Vert^2\\
&&\ \ \leq \frac{2(q-1)}{q^2}\Vert u_\varepsilon^{q/2}\Vert^2_{H^1}+C\Vert u_\varepsilon^{q/2}\Vert^2 \Vert \nabla v_\varepsilon\Vert^4_{L^6}+\!r\Vert u_\varepsilon^{q/2}\Vert^2.\label{B-8}
\end{eqnarray*}
Adding $\frac{2(q-1)}{q^2}\|u^{q/2}_\varepsilon\|^2$ to both sides,  and using that $v_\varepsilon$ is bounded in $L^\infty(H^2)$ (in particular $\nabla v_\varepsilon$ in $L^\infty(L^6)$), we have
\begin{equation*}\label{B-9}
\frac{d}{dt}\|u^{q/2}_\varepsilon\|^2+\|u^{q/2}_\varepsilon\|^2_{H^1}\le C_\varepsilon\|u^{q/2}_\varepsilon\|^2 .
\end{equation*}
Thus, from the Gronwall lemma, $u^{q/2}_\varepsilon$ is bounded in $W_2,$ for any $2\leq q<\infty$; hence in particular $u_\varepsilon$ is bounded in $A$. Therefore, the possible fixed points of $\lambda \mathcal{L}$ are bounded in $L^{\infty-}(Q)\times W_2$, independently of the parameter $\lambda>0$. Therefore, from Steps 1 and 2, the operator $\mathcal{L}$ satisfies the assumptions of the Leray-Schauder fixed point theorem. Hence, the mapping $\mathcal{L}$ has a fixed point $(u_\varepsilon,w_\varepsilon)\in A\times X_{5/2}$, which is a solution of problem \eqref{eq1_aprox}. 
\end{proof}

\subsection{Existence of weak solution (Proof of Theorem \ref{teo1}).} \label{Se:3.1}

\

The proof of Theorem \ref{teo1}   will be deduced  as  limit of the regularized system  \eqref{eq1_aprox} when $\varepsilon$ goes to $0$. 
 By energy inequality \eqref{B-7} (for $p\in (1,5/3]$ and non-logistic problem) or \eqref{Bb-5} (for $p>1$ and logistic problem),   the following bounds independent of $\varepsilon>0$ hold from the Gronwall Lemma:
\begin{equation}\label{P-1}
\left\{
\begin{array}{l}
\{ u_\varepsilon^{p/2}\}_{\varepsilon>0} \hbox{ in } L^\infty(L^2)\cap L^2(H^1)\hookrightarrow L^{10/3}(Q),\\
\{ v_\varepsilon\}_{\varepsilon>0} \hbox{ in } L^\infty(H^1)\cap L^2(H^2)\hookrightarrow L^{10}(Q),\\
\{ \sqrt\varepsilon\Delta v_\varepsilon \}_{\varepsilon>0} \hbox{ in } L^\infty(L^2)\cap L^2(H^1),\\
\{ u_\varepsilon\}_{\varepsilon>0} \hbox{ in } L^{2p-1}(Q)\ \mbox{(in the logistic case)}.
\end{array}
\right.
\end{equation}
Then, from \eqref{P-1}$_1$ we deduce that $\{u_\varepsilon^p\}_{\varepsilon>0}\mbox{ is bounded in }L^{5/3}(Q)$; which implies that 
\begin{equation}\label{P-2}
\{u_\varepsilon\}_{\varepsilon>0}\mbox{ is bounded in }L^{5p/3}(Q).
\end{equation}
In order to deduce a bound for $\nabla u_\varepsilon$, we consider two cases:

{\it \underline{Case $p\le 2$}.} 

By using the equality $\nabla u_\varepsilon=\frac2p u_\varepsilon^{1-p/2}\nabla(u_\varepsilon^{p/2})$ and the 
bounds $\nabla(u_\varepsilon^{p/2})$ in $L^2(Q)$ and \eqref{P-2},
  we have that $\nabla u_\varepsilon$ is bounded in $L^{5p/(3+p)}(Q)$.

{\it \underline{Case $p> 2$}.} 

In this case, we are going to estimate $u_\varepsilon^{q/2}$ in $W_2$ for $q\le q(p)=5p/6$. Indeed, by testing $u_\varepsilon$-equation by $u_\varepsilon^{q-1}$ and bounding the chemotaxis term as 
$$
\int_\Omega u_\varepsilon\nabla v_\varepsilon\cdot \nabla (u_\varepsilon^{q-1})
=- C(q) \int_\Omega u_\varepsilon^q\Delta v_\varepsilon
\le  C\, \|u_\varepsilon\|_{L^{2q}}^q \|\Delta v_\varepsilon\|
$$
which is bounded in $L^1(0,T)$ using that $\Delta v_\varepsilon$ is bounded in $L^2(Q)$ and $u_\varepsilon$ in $L^{2q}(Q)$ (recall that $q\le 5p/6$). Now we distinguish the case $2<p<12/5,$ where we take $q(p)<2;$ and the case $p\ge 12/5,$ where since $q(p)\ge 2,$ we will take  $q=2$ (hence $\nabla u_\varepsilon$ is bounded in $L^2(Q)$).  In the first case $2<p<12/5$,  using again the equality $\nabla u_\varepsilon=\frac2q u_\varepsilon^{1-q/2}\nabla(u_\varepsilon^{q/2})$, one can now deduce that $\nabla u_\varepsilon$ is bounded in $L^{25p/(18+5p)}(Q)$. In summary, 
one has
\begin{equation}\label{P-3}
\{\nabla u_\varepsilon\}_{\varepsilon>0}\mbox{ is bounded in }L^{\gamma(p)}(Q),
\end{equation}
where $\gamma(p)$ is defined in \eqref{st-11-bis}. In any case, from \eqref{P-1}$_2$, one has that $\nabla v_\varepsilon$ is bounded in $L^\infty(L^2)\cap L^2(H^1)\hookrightarrow L^{10/3}(Q)$; which joint with \eqref{P-2} imply that
\begin{equation}\label{P-4}
\{u_\varepsilon\nabla v_\varepsilon\}_{\varepsilon>0}\mbox{ is bounded in }L^{10p/(6+3p)}(Q).
\end{equation}
Now, from \eqref{eq1_aprox}$_1$ and the bounds \eqref{P-2}-\eqref{P-4} we can deduce that 
\begin{equation}\label{P-5}
\{\partial_tu_\varepsilon\}_{\varepsilon>0}\mbox{ is bounded in }
(L^{\mu(p)}(W^{1,10p/(7p-6)}))',
\end{equation}
where $\mu(p)$ is defined by
\begin{equation} \label{mu(p)}
\mu(p)=
\left\{ 
\begin{array}{ll}
10p/(7p-6) & \hbox{if $p\in (1,2]$,}\\
 5/2 & \hbox{if $p\in (2,+\infty)$.}
\end{array}
\right.
\end{equation} 
 
On the other hand, from \eqref{P-1}$_{2,3}$ and taking into account that $w_\varepsilon=v_\varepsilon-\varepsilon\Delta v_\varepsilon$, we have 
\begin{equation}\label{P-6}
\{w_\varepsilon\}_{\varepsilon>0}\mbox{ is bounded in }L^\infty(L^2)\cap L^2(H^1).
\end{equation}
Also, considering \eqref{eq1_aprox}$_2$, \eqref{P-1}$_2$, \eqref{P-2} and \eqref{P-6} we conclude that
\begin{equation}\label{P-7}
\{\partial_tw_\varepsilon\}_{\varepsilon>0}\mbox{ is bounded in }L^{5/3}((H^1)').
\end{equation}
Now, we observe that using the equality $w_\varepsilon=v_\varepsilon-\varepsilon\Delta v_\varepsilon$ and \eqref{P-1}$_3$ we obtain
\begin{equation}\label{P-8}
w_\varepsilon-v_\varepsilon=-\varepsilon\Delta v_\varepsilon\to0,\mbox{ as }\varepsilon\to0,\mbox{ in }L^\infty(L^2)\cap L^2(H^1).
\end{equation}
Hence, from \eqref{P-1}$_2$, \eqref{P-2} and \eqref{P-3}, we deduce that there exists a limit  pair $(u,v)$ such that 
$$
\left\{
\begin{array}{l}
u\in L^{5p/3}(Q),\quad  \nabla u \in L^{\gamma(p)}(Q),\\
v\in L^\infty(H^1)\cap L^2(H^2),
\end{array}
\right.
$$
and a subsequence of $\{(u_\varepsilon,v_\varepsilon,w_\varepsilon)\}_{\varepsilon>0}$, still denoted by  $\{(u_\varepsilon,v_\varepsilon,w_\varepsilon)\}_{\varepsilon>0}$, satisfying the following convergences, as $\varepsilon\to0$:
\begin{equation}\label{P-9}
\left\{
\begin{array}{rcl}
u_\varepsilon&\to&u\mbox{ weakly in }L^{5p/3}(Q),\\
\nabla u_\varepsilon&\to&\nabla u\mbox{ weakly in }L^{\gamma(p)}(Q),\\
u_\varepsilon&\to&u\mbox{ weakly in }L^{2p-1}(Q)\ \mbox{(in the logistic case),}\\
v_\varepsilon&\to&v\mbox{ weakly in }L^2(H^2)\mbox{ and weakly* in }L^\infty(H^1),\\
w_\varepsilon&\to&v\mbox{ weakly in }L^2(H^1)\mbox{ and weakly* in }L^\infty(L^2),\\
\partial_tu_\varepsilon&\to&\partial_tu\mbox{ weakly in } (L^{\mu(p)}(W^{1,10p/(7p-6)}))',\\
\partial_tw_\varepsilon&\to&\partial_tv\mbox{ weakly in }L^{5/3}((H^1)').
\end{array}
\right.
\end{equation}
From \eqref{P-3}, \eqref{P-5} and the Aubin-Lions compactness lemma, we have 
\begin{equation}\label{P-10}
\{u_\varepsilon\}_{\varepsilon>0}\mbox{ is relatively compact in }L^{\gamma(p)}(B), 
\end{equation}
for any Banach space  $B$ such that $W^{1,\gamma(p)}(\Omega)\hookrightarrow B \hookrightarrow (W^{1,10p/(7p-6)})'$, where the first embedding is compact. In particular, we can consider $B=L^{15/7}(\Omega)$ (using that $W^{1,5/4}(\Omega)\hookrightarrow L^{15/7}(\Omega)$ and $\gamma(p)>5/4$ for all $p>1$). Therefore, using this compactness jointly to the estimate \eqref{P-2}, there exists 
 a subsequence of $\{u_\varepsilon\}_{\varepsilon>0}$ (equally denoted) and a limit function $u\ge 0$ in $Q$ and $u\in L^{5p/3}(Q)$ such that  
\begin{equation}\label{P-12}
\left\{
\begin{array}{l}
u_\varepsilon^p\to u^p\mbox{ strongly in }L^{q}(Q)\ \forall q<5/3,\\
u_\varepsilon^p\to u^p\mbox{ weakly in }L^{1+(p-1)/p}(Q).
\end{array}
\right.
\end{equation}
On the other hand, convergences (\ref{P-9})$_3$ and (\ref{P-9})$_6$ and the Aubin-Lions compactness imply 
\begin{equation}\label{P-12b}
w_\varepsilon\to v\mbox{ strongly in }L^{2}(Q)\cap C((H^1)').
\end{equation}
From (\ref{P-4}), (\ref{P-9})$_2$ and (\ref{P-10}) we have that $u_\varepsilon\nabla v_\varepsilon\to u\nabla v$ weakly in $L^{10p/(6+3p)}(Q).$ Also, (\ref{P-8}), (\ref{P-9})$_2$ and (\ref{P-12b}) implies that ${v_\varepsilon}_+\to v_+$ strongly in $L^2(Q).$ In addition, from (\ref{P-1})$_2$ and since  $f\in L^{5/2}(Q_c),$ it holds that $f{v_\varepsilon}_+1_{\Omega_c}\to fv_+1_{\Omega_c}$ weakly in $L^{2}(Q).$ 
From previous convergences we can pass to the limit in the regularized system \eqref{eq1_aprox}, as $\varepsilon\to0$, and we conclude that the limit par  $(u,v)$ satisfies the variational problem 
\begin{eqnarray}
\int_0^T\langle\partial_tu,\overline{u}\rangle+\int_0^T(\nabla u,\nabla\overline{u})+\int_0^T(u\nabla v,\nabla\overline{u})=
\int_0^T(r\, u-\mu\, u^p,\overline{u}),\label{P-13}\\
\int_0^T\langle\partial_tv,\overline{w}\rangle+\int_0^T(\nabla v,\nabla\overline{w})+\int_0^T(v,\overline{w})
=\int_0^T(u^p,\overline{w})+\int_0^T(fv_+1_{\Omega_c},\overline{w}),\label{P-14}
\end{eqnarray}
for any $\overline{u}\in L^{\mu(p)}(W^{1,10p/(7p-6)})$ and 
$\overline{w}\in L^{5}(H^1)$. 
 Moreover, from \eqref{P-14}, considering that $u^p\in L^{5/3}(Q)$, $v\in L^2(H^2)$, and  integrating by parts in $\Omega$, we deduce the $v$-equation \eqref{eq1}$_2$ is satisfied   pointwisely a.e. $(t,x)\in Q$, i.e.,
\begin{equation}\label{P-15}
\partial_tv-\Delta v+v=u^p+fv_+1_{\Omega_c},\ \mbox{a.e.}\  \mbox{ in $Q$}.
\end{equation}
We already have that $u\ge0$ a.e. in $Q$. The positivity of $v$ follow by testing \eqref{P-15} by $v_-:=\min\{v,0\}\le0$ and using that $v_-=0$ if $v\ge0$, $\nabla v_-=\nabla v$ if $v\le0$ and $\nabla v_-=0$ if $v>0$, obtaining that $v_-\equiv0$. Consequently, $v\ge0$ a.e. in $Q$.  

On the other hand,  from $(\ref{P-1})_1$ and (\ref{P-5}), one has 
\begin{equation}\label{P-12bis}
u_\varepsilon\to u\mbox{ strongly in }  C((W^{1,10p/(7p-6)})').
\end{equation}
Using (\ref{P-12bis}), (\ref{P-12b}) and the assumption (\ref{elip1}) one has 
$$(u_\varepsilon(0),w_\varepsilon(0)) \to (u_0,v_0)\quad  \hbox{in $(W^{1,10p/(7p-6)} (\Omega))'\times (H^1(\Omega))'$}.$$
 In particular,  $(u(0),v(0))= (u_0,v_0)$ in $L^p(\Omega)\times H^1(\Omega)$. Consequently, the pair $(u,v)$ is a weak solution of system \eqref{eq1} (for $p\in(1,5/3]$) or \eqref{eq1b} (for $p\in(1,+\infty)$).


\subsection{Regularity Criterion for strong solution (Proof of Theorem \ref{strong}).} \label{seccionCriterium}

\

 
To prove Theorem \ref{strong} 
 we will  introduce some decoupled regularized problems 
 where all computations can be rigorously justified. Let $(u,v)$ be any weak solution of the system   (\ref{eq1}) or (\ref{eq1b}) satisfying the regularity  criterion \eqref{reg-crit}. 
The proof is divided in three steps:\\

 {\it \underline{Step 1: A non-hilbertian strong regularity for $v$: $v\in X_{5/2}$}}. 
 
 \
 
  If  $f\in L^{5/2}(L^{5/2+}(\Omega_c))$  and $u^p\in L^{5/2}(Q)$, 
by applying Corollary \ref{lem-A} (see Appendix \ref{AppendixA}, below) with $g=u^p\in L^{5/2}(Q)$ and $a=1-f\, 1_{\Omega_c}\in L^{5/2}(L^{5/2+}(\Omega)) $, we obtain the existence of $w\in X_{5/2}$  solving
\begin{eqnarray}\label{uni1}
\partial_tw-\Delta w+w=u^p+fw\, 1_{\Omega_c},\quad w(0)=v_0,
\quad \partial_{\bf n} w |_{\partial\Omega}=0.
\end{eqnarray}

On the other hand, the weak solution $v\in X_2$ is also a solution of the linear problem (\ref{uni1}). Then, by uniqueness of (\ref{uni1}) (in $W_2$)
one has   $v=w\in X_{5/2}$ (and in particular, by Lemma~\ref{An}, $\nabla v \in L^5(Q)$). Moreover, there exists 
$
K_1:=K_1(\|v_0\|_{W^{6/5,5/2}},\|f\|_{L^{5/2}(L^{5/2+})}, \|u ^p\|_{L^{5/2}(Q)})>0
$ such that
\begin{equation}\label{xx-1}
\|v\|_{X_{5/2}}\le K_1.
\end{equation}

 {\it \underline{Step 2: Additional regularity for $u$: $u\in X_u$}}.

 Since $u\in L^{5p/2}(Q)
 $, then $u^{p-1}\in L^{5p/(2(p-1))}(Q)\subset L^{5/2}(Q)$. Thus, we can use Lemma~\ref{lem-B}, by taking $w_0=u_0\in L^{\alpha(p)}(\Omega)$, ${\bf c}=\nabla v \in L^5(Q)$, $a=\mu \, u^{p-1}-r\in L^{5/2}(Q),$ $g^0=0$ and ${\bf g}^1={\bf 0}.$ Then, there exists a unique solution $\widetilde u$  of the  linear problem
 \begin{equation}\label{lin-aux}
 \left\{
 \begin{array}{rcl}
 \partial_t  \widetilde u-\Delta \widetilde u-\nabla\cdot(\widetilde u \nabla v) +(\mu\, u^{p-1}-r) \widetilde u&=& 0\ \mbox{ in }\ Q,\\
 \widetilde u(0)&=&u_0\ \mbox{ in }\ \Omega,\\
 (-\nabla \widetilde u - \widetilde u \nabla v)\cdot{\bf n}&=&0\ \mbox{ on }\ (0,T)\times\partial\Omega,
 \end{array}
 \right.
 \end{equation}
 such that  $\widetilde u^{q/2}\in W_2$ for any $q: 2\le q \le \alpha(p).$ In addition, for $q= \alpha(p)$ one also has $\widetilde u\in L^{\beta(p)}(Q)$ (because $\widetilde u ^{\alpha(p)/2}\in W_2\subset L^{10/3}(Q)$ and $\beta(q)=(5/3)\alpha(p)$); and for $q= 2$ one has $\nabla \widetilde u\in L^2(Q)$.
Consequently $\widetilde u\in X_u$ and there exists $\widetilde{K}_2=\widetilde{K}_2(\|u_0\|_{L ^{\alpha(p)}}, \Vert \nabla v \Vert_{L^{5}(Q)}, \Vert u^{p-1}\Vert_{L^{5/2}(Q)})>0$ such that 
\begin{equation}\label{xx-2}
\|\widetilde{u}\|_{X_u}\le \widetilde{K}_2.
\end{equation}

Therefore, it suffices to prove that $u=\widetilde u.$ We argue by duality,  by considering the dual problem of \eqref{lin-aux} for any second member $\varphi\in L^{5\alpha(p)/(2\alpha(p)+3)}(Q)$: to find $z$ solving
 \begin{equation}\label{dual-lin}
 \left\{
 \begin{array}{rcl}
- \partial_t z-\Delta z- \nabla v\cdot \nabla z+(\mu\, u^{p-1}-r) z&=& \varphi\ \mbox{ in }\ Q,\\
 z(T)&=&0\ \mbox{ in }\ \Omega,\\
 -\nabla z\cdot{\bf n}&=&0\ \mbox{ on }\ (0,T)\times\partial\Omega.
 \end{array}
 \right.
 \end{equation}
 Following the same argument of  Lemma~\ref{lem-B}, there exists a unique solution 
 $z\in L^\infty (L^{\alpha(p)})\cap L^{\beta(p)}(Q)$ with $\nabla z\in L^2(Q)$ ($z\in X_u$).  Since $u$ also satisfies  the problem \eqref{lin-aux} (only in a weeker sense that $\widetilde u$), then  testing the  problem \eqref{lin-aux} satisfied by $u-\widetilde u$ by  $z$ (the solution of dual problem \eqref{dual-lin}), 
 which is possible owing to the hypothesis $u\in L^{5p/2}(Q)\cap L^{10/3}(Q)$, one has 
$$
0=\int_0^T\int_\Omega(u-\widetilde u)\left(- \partial_t z-\Delta z- \nabla v\cdot \nabla z+(\mu\, u^{p-1}-r) z \right)
=\int_0^T\int_\Omega(u-\widetilde u)\varphi.
$$
Since that is true for any $\varphi\in  L^{5\alpha(p)/(2\alpha(p)+3)}(Q)$, one has $u=\widetilde u$. Consequently, $u\in X_u$ and there exists a constant
${K}_2={K}_2(\|u_0\|_{L ^{\alpha(p)}}, \Vert \nabla v \Vert_{L^{5}(Q)}, \Vert u^{p-1}\Vert_{L^{5/2}(Q)})>0$ such that 
\begin{equation}\label{xx-2}
\|u\|_{X_u}\le K_2.
\end{equation}
The estimate \eqref{bound_sol} follows from \eqref{xx-1} and \eqref{xx-2}.

\

 {\it \underline{Step 3: Uniqueness of strong solutions.}}
 
 \
 
Following a classical comparison argument and using the Gagliardo-Nirenberg 
interpolation inequality (see, for instance, \cite{nirenberg}) and Gronwall Lemma, we can deduce the uniqueness of strong solutions
of problem (\ref{eq1})-(\ref{eq2}) or (\ref{eq1b})-(\ref{eq2}).  In fact, if $(u_1,v_1), (u_2,v_2)\in X_u\times X_v$ are two possible solutions of  (\ref{eq1})-(\ref{eq2}) or (\ref{eq1b})-(\ref{eq2})
related to the same $f$, then, subtracting the respective equations for $(u_1,v_1)$ and $(u_2,v_2)$ and denoting $(u,v):=(u_1-u_2,v_1-v_2),$ we can deduce that the pair $(u,v)$ satisfies the following system
\begin{equation}\label{U-1}
\left\{
\begin{array}{rl}
&\displaystyle\int_0^T\langle\partial_tu,\overline{u}\rangle_{(H^1)'}
+\int_0^T\int_\Omega\nabla u\cdot\nabla\overline{u}
-\displaystyle\int_0^T\int_\Omega(u\nabla v_1+u_2\nabla v)\cdot\nabla\overline{u}\\
&\ \ \ \ \ \ \ \ \ = \displaystyle\int_0^T\int_\Omega\big(ru-\mu(u_1^p-u_2^p)\big)\overline{u},\ \ \forall\overline{u}\in  Z=Y_u',\\
\noalign{\vspace{-1ex}}\\
 &\partial_tv-\Delta v+v=u_1^p-u_2^p+fv\,1_{\Omega_c}\ \mbox{ a.e. }(t,x)\in Q,\\
 &u(0)=0, \ \ \ v(0)=0,\ \mbox{ in }\ \Omega,\\
 &\partial_{\bf n}v=0,\ \mbox{ on }\ (0,T)\times\partial\Omega.
\end{array}
\right.
\end{equation}
Moreover, from the mean-value theorem we deduce that there exists a measurable and non-negative function $\varphi:Q\to [0,\infty)$ attaining  values between $u_1$ and $u_2$ such that
$$
u_1^p(t,x)-u_2^p(t,x)=p\,\varphi^{p-1}(t,x)u(t,x)\ \mbox{ a.e. }(t,x)\in Q.
$$
Thus,
\begin{equation}\label{U-1-1}
-\mu\int_0^T\int_\Omega(u_1^p-u_2^p)\overline{u}
=-\mu p\int_0^T\int_\Omega\varphi^{p-1}u\overline{u}, \quad \forall \overline{u}\in Z=Y'_u.
\end{equation}

Then, making $\overline{u}=u\in X_u\subset Z$ in \eqref{U-1}$_1$ and testing \eqref{U-1}$_2$ by $(v-\Delta v)\in L^{5/2}(Q)\subset L^{5/3}(Q)$, using the H\"older and Young inequalities, the point-wise inequality $|x^p-y^p|\le |x-y|(|x|^{p-1}+|y|^{p-1})$ for $x,y\ge0$, and taking into account the sign of (\ref{U-1-1}) when $\overline{u}=u,$ and the classical 3D interpolation inequalities 
$\|u\|_{L^{10/3}}\le C\|v\|^{2/5}\|u\|^{3/5}_{H^1}$ and 
$\|v\|_{L^{10}}\le C\|v\|_{H^1}^{4/5}\|v\|^{1/5}_{H^2}$, we arrive at the time differential inequality
\begin{eqnarray}\label{U-3}
&&\frac{d}{dt}\left(\|u\|^2+\|v\|^2_{H^1}\right)+C\left(\|u\|^2_{H^1}+\|v\|^2_{H^1}+\|v\|^2_{H^2}\right)\nonumber\\
&&\  \ \le C\!\left(\|u_1\|_{L^{5(p-1)}}^{5(p-1)}+\|u_2\|_{L^{5(p-1)}}^{5(p-1)}
+\|\nabla v_1\|_{L^5}^5+1\right)\|u\|^2\nonumber\\
&&\ \  \ \ +C\left(\|u_2\|^5_{L^5}+\|f\|^{5/2}_{L^{5/2}}\right)\|v\|^2_{H^1}.
\end{eqnarray}

Since the norms 
$\|u_2\|^5_{L^5}, $  $\|u_i\|^{5(p-1)}_{L^{5(p-1)}} $ and $\|\nabla v_1\|_{L^5}^5$   are integrable in time
 thanks to \eqref{xx-2} and \eqref{xu}-\eqref{xv}, then   from \eqref{U-3}, the Gronwall Lemma and the fac that $(u(0),v(0))=(0,0),$ we arrive at $(u,v)=(0,0)$. Consequently $(u_1,v_1)=(u_2,v_2)$, which conclude the uniqueness.

The proof of Theorem \ref{strong} is then finished. \hfill$\Box$

\


At this point, we introduce some remarks about the previous proof. 

\begin{remark}\label{r:3.1}
 Analyzing 
 the regularity for $\nabla \cdot (u \nabla v)$, we get
that $u \in L^{5}(Q)$  and $\Delta v \in L^{5/2}(Q)$ implies $u \,\Delta v \in L^{5/3}(Q),$ and 
 $\nabla v \in L^5(Q)$ and $\nabla u \in L^{2}(Q)$ implies $\nabla u \cdot \nabla v \in L^{10/7}(Q).$ 
Therefore, $\nabla \cdot (u \nabla v)=u\Delta v+\nabla u\cdot\nabla v \in L^{10/7}(Q)$ and, assuming $u_0\in W^{3/5,10/7}(\Omega)$, one has $u\in X_{10/7}$ by parabolic  regularity of the heat equation.  In particular, this implies $u \in L^{10/3}(Q)$ and $\nabla u \in L^2(Q)$ by Lemma \ref{An}; thus, to get more regularity than $u\in X_{10/7}$ is not possible by means of  a bootstrap argument. This is the reason why we consider the $u$-equation only  in a variational sense.
\end{remark}

\begin{remark}
By assuming a little more regularity than \eqref{reg-crit}, that is  $u^p \in L^{5/2+}(Q)$ and $f\in L^{5/2+}(Q_c)$, from  parabolic regularity, 
$
v \in X_{5/2+}\subset L^\infty(Q)
$
(cf. Lemma \ref{An}). Then, following the bootstrap argument of \cite[Theorem 3.3]{Andre_SICOM} based on the heat-Neumann problem,  it also holds $u \in X_{5/2+}$.
Now, since we only assume $f\in L^{5/2}(L^{5/2+}(\Omega_c))$,  we have arrived at $v\in X_{5/2}$ but with this regularity the bootstrap argument on the heat equation to improve regularity on $u$ does not work (see Remark \ref{r:3.1}).  This is the reason why we consider the auxiliary linear problem \eqref{lin-aux} to improve the regularity of $u$, where the convective and the reaction terms are incorporated into the linear operator of \eqref{lin-aux}.
\end{remark}

\begin{remark} Only assuming $f\in L^{5/2}(Q_c)$ and $u \in L^{5p/2}(Q)\cap L^{10/3}(Q)$ 
is not sufficient to get strong regularity for $u$. 
Indeed, note that Step 1 of Theorem \ref{strong} is still true. However, following Step 2, now  $f\in L^{5/2}(Q_c)$ and $v \in L^{\infty}(L^{\infty-})$; hence 
$f v\,1_{\Omega_c} \in L^{5/2}(L^{5/2-})$. As $u^p \in L^{5/2}(Q)$, then the parabolic  regularity only gives
$v \in X_{5/2-}.$ Therefore $\nabla v\not\in L ^5(Q)$ in general. However, the regularity $\nabla v\in L^5(Q)$ has been essential in the  argument made in Step 2 and even more, it will be also essential for the well-posedness of the linearized problem (see Subsection \ref{linealizado} below).
\end{remark}
\section{Optimal control. Existence of optimal solutions and optimality conditions}
\subsection{Existence of optimal solution (proof of Theorem \ref{control}).} \label{Sec:OptimalSolution}
\begin{proof}
By assumption $\mathcal{S}_{ad}$ is not empty. Furthermore, since  $J$ is bounded from below, there exists a  minimizing sequence $\{s_m\}_{m\in\mathbb{N}}:=\{(u_m,v_m,f_m)\}_{m\in\mathbb{N}}\subset \mathcal{S}_{ad}$ such that
\begin{gather*}
    \lim_{m\to\infty}J(s_m)=\inf_{s\in \mathcal{S}_{ad}} J(s) (\ge 0).
\end{gather*}

From the definition of $\mathcal{S}_{ad}$, for all $m\in\mathbb{N},$ $\{u_m,v_m,f_m\}_{m\in\mathbb{N}}$ satisfies  (\ref{eq1}) or  (\ref{eq1b}). 
 From the definition of $J$ and  the assumption  $\gamma_f>0$ or $\mathcal{F}$ is bounded in 
$L^{5/2}(L^{5/2+}(\Omega_c))$, we conclude that $\{f_m\}_{m\in\mathbb{N}}$ is bounded in $L^{5/2}(L^{5/2+}(\Omega_c))$ and $\{(u_m)^p\}_{m\in\mathbb{N}}$ is bounded in $L^{5/2}(Q)\cap L^{10/3p}(Q).$ 

Therefore, from Theorem~\ref{strong}, 
the sequence $\{(u_m,v_m)\}_{m\in\mathbb{N}}$  is bounded in 
$X_u\times X_v$. 
Then, recalling that $\mathcal{F}$ is a closed and convex subset of $L^{5/2}(L^{5/2+}(\Omega_c))$
 (thus, weakly closed in $L^{5/2}(L^{5/2+}(\Omega_c))$),
  there exists a subsequence of $\{s_m\}_{m\in\mathbb{N}},$ still denoted by $\{s_m\}_{m\in\mathbb{N}}$ and a limit $(u^*,v^*,f^*)$ with $f^*\in\mathcal{F}$ such that

\begin{equation}\label{convergencia}
\left\{
\begin{array}{rcll}
    (u_m,v_m)&\rightarrow&(u^*,v^*)&\text{weakly in}\ 
   L^{\beta(p)}(Q)\times L^{5/2}(W^{2,5/2}),
   \\
    \left(u_m,v_m\right)&\rightarrow&(u^*,v^*) &\text{weakly${*}$ in}\
   L^{\infty}(L^{\alpha(p)})\times L^{\infty}(W^{6/5,5/2}),
\\
    \left(\partial_t u_m, \partial_t v_m\right)&\rightarrow& (\partial_t u^*,\partial_t v^*)
    &\text{weakly in}\      L^{2}((H^1)')\times L^{5/2}(Q), \\
      f_m&\rightarrow& f^*\ &\text{weakly in}\ L^{5/2}(L^{5/2+}(\Omega_c)) .
\end{array}
\right.
\end{equation}

In particular, from \eqref{convergencia}   one has the compactness (\cite[Theorem 5.1]{Lions} and \cite[Corollary 4]{Simon})
\begin{eqnarray}
u_m&\to& u^*\mbox{ strongly in }
L^2(Q)\cap C((H^1)'),\\\label{conv-x}
v_m&\to& v^*\mbox{ strongly in }
C(L^{\infty-}).\label{conv-x1}
\end{eqnarray}
By using the bound of $u_m$ in $L^{\beta(p)}(Q)$, one has 
$$
u_m\to u^*\mbox{ strongly in } L^q(Q) \quad \forall\, q<\beta(p).
$$
Now, since $\{v_m\}_{m\in\mathbb{N}}$ is bounded in $X_v$ we deduce that 
$\{\nabla v_m\}_{m\in\mathbb{N}}$ is bounded in $L^5(Q)$ and 
$$
\nabla v_m \to \nabla v^* \mbox{ weakly in } L^5(Q).
$$
From the above convergences 
we get
\begin{equation}\label{con-x2}
\left\{
\begin{array}{rcll}
u_m\nabla v_m&\to&u^*\nabla v^* &\mbox{weakly in } L^{5/2}(Q), \\
f_mv_m\,1_{\Omega_c}&\to&f^*v^*\,1_{\Omega_c} &\mbox{weakly in }L^{5/2}(Q).
\end{array}
\right.
\end{equation}

From (\ref{conv-x})-(\ref{conv-x1}), it holds $(u_{m}(0),v_{m}(0))\rightarrow(u^*(0),v^*(0))\ \ \text{strongly in}\ 
(H^1(\Omega))' \times L^{\infty-}(\Omega)$,
and since $u_{m}(0)=u_0,\ v_{m}(0)=v_0$ for all $m\in\mathbb{N}$, then 
$(u^*(0),v^*(0))=(u_0,v_0).$ 

Then, considering the convergences (\ref{convergencia})-(\ref{con-x2}) we get that 
$s^*=(u^*,v^*,f^*)$ is the strong solution of problem  (\ref{eq1}) or  (\ref{eq1b}) and the initial condition $(u^*(0),v^*(0))=(u_0,v_0)$, that is, $s^*\in \mathcal{S}_{ad}.$ Thus, 
\begin{gather}\label{eq013}
    \lim_{m\to\infty}J(s_m)=\inf_{s\in \mathcal{S}_{ad}} J(s)\leq J(s^*).
\end{gather}
Finally, since $J$ is weakly lower semicontinuous on 
$\left(L^{5p/2}(Q)\!\cap L^{10/3}(Q)\right)\!\times L^2(Q) \times L^{5/2}(L^{5/2+}(\Omega_c)),$ we get 
\begin{gather}\label{eq014}
    J(s^*)\leq\liminf_{m\to\infty} J(s_m).
\end{gather}
From (\ref{eq013}) and (\ref{eq014}) we conclude (\ref{os-1}).
\end{proof}

\begin{remark}\label{R4-1}
We have imposed the hypothesis $\mathcal{S}_{ad}\ne\emptyset$ because there is no guarantee about the existence of global strong solutions of system \eqref{eq1}-\eqref{eq2} (or \eqref{eq1b}-\eqref{eq2}) defined in $(0,T)$. 
Following the ideas of \cite[Remark 6]{Exequiel3D} and \cite{Lopez}, we can furnish a particular case where this hypothesis holds, namely, when $\Omega_c=\Omega$ and the initial data 
$u_0, v_0 \in W^{6/5+,5/2+}(\Omega)$ such that $0<\kappa\le v_0$ in $\Omega$. Indeed, from \cite[Theorem 10.22]{feireisl} we consider the unique strong solution $ v\in X_{5/2+}$ of the Neumann problem 
$$
\partial_tv-\Delta v=0\mbox{ in }Q,\ v(0,\cdot)=v_0\mbox{ in }\Omega,\ {\partial}_{\bf n}v=0\mbox{ on } (0,T)\times\partial\Omega.
$$
This solution also satisfies $v\ge \kappa>0$ in $Q$. 
Now, giving this positive function $v\in X_{5/2+}$ we consider the parabolic problem 
\begin{equation}\label{nv-1}
\left\{
\begin{array}{rcl}
\partial_tu-\Delta u-\nabla\cdot(u\nabla v)&=&r\,u-\mu\, u^p\ \mbox{ in $Q$,}\\
u(0)&=&u_0\ \mbox{ in $\Omega$,}\\
\partial_{\bf n}u&=&0\ \mbox{ on $(0,T)\times\partial\Omega$.}
\end{array}
\right.
\end{equation}

Then, testing (\ref{nv-1})$_1$ by $|u|^{\alpha(p)-2}u$ the following inequality can be obtained (see  Lemma \ref{lem-B}):
$$
\frac{d}{dt}\| |u|^{\alpha(p)/2}\|^2+C\| |u|^{\alpha(p)/2} \|^2_{H^1}+\Vert u\Vert_{L^{\alpha(p)+p-1}}^{\alpha(p)+p-1}\le C\left(1+\|\nabla v\|^5_{L^5}\right)\| |u|^{\alpha(p)/2} \|^2;
$$
which implies $|u|^{\alpha(p)/2}\in L^\infty(L^2)\cap L^2(H^1)\hookrightarrow L^{10/3}(Q)$, and then $u\in L^{\beta(p)}(Q)$.
By applying again Lemma \ref{lem-B} below, one has that $\nabla u\in L^2(Q)$. 
Since  $v\in X_{5/2+},$ then $\nabla v\in L^{5+}(Q)$  that jointly to the regularity $\nabla u\in L^2(Q)$ gives that $\nabla\cdot(u\nabla v)\in L^{10/7+}(Q).$
 Then, following the bootstrapping argument of \cite[Theorem 3.3]{Andre_SICOM}, 
 one arrives at $u\in X_{5/2+}$, which implies in particular that  $u\in L^\infty(Q)$ by Lemma \ref{An}. 
 Finally, since the control acts on the whole domain $\Omega$ (that is, $\Omega_c=\Omega$),  
 we can take the control $f$ as the solution of the point-wise equation 
$v=u^p+f\, v$ in $Q$, 
that is, $f=1-u^p/v \in L^\infty(Q)$,  which is an admisible control, and thus,  $\mathcal{S}_{ad}\ne\emptyset$.


\end{remark}

 \subsection{Differentiability of the equality operator}

 \
 
 
We consider the equality operator defined in (\ref{oper-s}):
$$\mathcal{S}:=(\mathcal{S}_1,\mathcal{S}_2):
{X}_{u}  \times {X}_{v} \times L^{5/2}(L^{5/2+}(\Omega_c))
\to Y_u \times Y_v. 
$$
Since $\mathcal{S}_{ad}\neq \emptyset$,  
let $(\widehat{u},\widehat{v},\widehat{f})$ be an element of $\mathcal{S}_{ad}$. We consider the spaces 
$$\widehat{X}_u=\{w \in X_u: \, w(0)=0\},
\quad 
\widehat{X}_v=\{w \in X_v: \ w(0)=0\}
$$
and  the closed and convex subset  of $X_u\times X_v\times L^{5/2}(L^{5/2+}(\Omega_c))$: 
$$\mathbb{M}=(\widehat{u},\widehat{v},\widehat{f})+\widehat{X}_{u}  
\times
\widehat{X}_{v}
\times
(L^{5/2}(L^{5/2+}(\Omega_c))-\widehat{f}).$$
Thus, the bilinear optimal control problem \eqref{C-3} is rewritten as
\begin{eqnarray}\label{sss-1}
\min_{ \left(u,v,f\right)\in \mathbb{M}}J(u,v,f)\ \mbox{ subject to }\ \mathcal{S}(u,v,f)={\bf 0}.
\end{eqnarray}

\

At this point, we deduce the  differentiability of the equality operator $\mathcal{S}$.
\begin{lemma}\label{der-oper}
The operator $\mathcal{S},$ from ${X}_{u}  \times {X}_{v} \times L^{5/2}(L^{5/2+}(\Omega_c))$ to $Y_u \times Y_v,$ is continuously Fr\'echet-differentiable and its derivative at an arbitrary point 
$z^*=(u^*,v^*,f^*)\in {X}_{ u} \times {X}_{v} \times L^{5/2}(L^{5/2+}(\Omega_c)),$ in the direction $z=(U,V,F)\in {X}_{ u} 
\times{X}_{v}\times L^{5/2}(L^{5/2+}(\Omega_c)) $ is the linear and bounded operator $\mathcal{S}'(z^*)=(\mathcal{S}_1'(z^*),\mathcal{S}_2'(z^*)): {X}_{ u} 
\times{X}_{v}\times L^{5/2}(L^{5/2+}(\Omega_c)) \to Y_u \times Y_v$ defined by
\begin{eqnarray}\label{sss-2-u}
&&\langle\mathcal{S}_1'(z^*)[z],\overline{u}\rangle_{Z'}
=
\int_0^T\langle\partial_tU,\overline{u}\rangle_{Z'}
+\int_0^T\int_\Omega\nabla U\cdot\nabla\overline{u}
\nonumber\\
&&+\int_0^T\int_\Omega\left(u^*\nabla V+U\nabla v^*\right)\cdot\nabla\overline{u}
-\int_0^T\int_\Omega\left(rU - \mu \,p \, |u^*|^{p-2} u^* U\right)\overline{u},
\quad \forall\,\overline{u}\in Z,
\end{eqnarray}
and 
\begin{equation} \label{sss-2-v}
\mathcal{S}_2'(z^*)[z]=
\partial_tV-\Delta V+V-p \, |u^*|^{p-2} u^* U-f^*V\,1_{\Omega_c}-Fv^*\,1_{\Omega_c}
\quad \hbox{in $L^{5/2}(Q)$}.
\end{equation}
\end{lemma}
\begin{proof}
%
We first analyze the Fr\'echet differentiability of $\mathcal{S}_1$.  Let $z^*=(u^*,v^*,f^*)$ and $z=(U,V,F)$ two elements of
${X}_u\times {X}_v\times L^{5/2}(L^{5/2+}(\Omega_c))$. We have
\begin{equation}\label{sss-3}
\left\{
\begin{array}{rcl}
\mathcal{S}_1(z^*+z)-\mathcal{S}_1(z^*)
&=&\partial_tU-\Delta U-\nabla\cdot(u^*\nabla V)-\nabla\cdot(U\nabla v^*)-\nabla\cdot(U\nabla V)\\
&& -rU + \mu\Big( |u^*+U|^p - |u^*|^p \Big) \ \ \mbox{ in } Y_u=Z'.
\end{array}
\right.
\end{equation}

Using that the real function  $\varphi(y)= |y|^p$ is continuously-differentiable in $\mathbb{R}$ (recall that $p>1$), from the mean-value theorem  there exists a measurable function $\theta:Q\rightarrow \mathbb{R}$ attaining intermediate values between the ones of $u^*$ and $u^*+U$ such that
\begin{equation}\label{sss-4}
|u^*(t,x)+U(t,x)|^p-|u^*(t,x)|^p=  p \, |\theta(t,x)|^{p-2}\theta(t,x) U(t,x),\ a.e.\ (t,x)\in Q.
\end{equation}
Furthermore, by the dominated convergence theorem, we have 

\begin{equation}\label{sss-4-1}
|\theta|^{p-2}\theta\to |u^*|^{p-2} u^* \hbox{ in $L^{\beta(p)/(p-1)}(Q)$, }
\mbox{ as } U\to 0 \hbox{ in $L^{\beta(p)}(Q)$}.
\end{equation}

Then, equality \eqref{sss-2-u} together with \eqref{sss-3}
 and \eqref{sss-4}  and Proposition \ref{ident}
imply

\begin{equation}\label{sss-5}
\begin{array}{rcl}
&&\|\mathcal{S}_1(z^*+z)-\mathcal{S}_1(z^*)-\mathcal{S}_1'(z^*)[z]\|_{Y_u}
\\
 && \quad \le  \|\nabla\cdot(U\nabla V)\|_{Y_u}
 +\mu p\,\|( |u^*|^{p-2} u^* - |\theta|^{p-2}\theta )U\|_{Y_u}\\
 &&\quad 
 \le \|U\nabla V\|_{L^{5\alpha(p)/(\alpha(p)+3)}(Q)}
 +\mu p\,\|( |u^*|^{p-2} u^* - |\theta|^{p-2}\theta )U\|_{L^{5\alpha(p)/(2\alpha(p)+3)}(Q)}.
\end{array}
\end{equation}

 Now, since $(U, V)\in X_u\times X_v$, then in particular $(U,\nabla V)\in L^{\beta(p)}(Q)\times L^5(Q)$. 
 Moreover, taking into account that $(5/3)\alpha(p)=\beta(p)$ for any $p>1$, 
 from the H\"older inequality we obtain
\begin{eqnarray}
\|U\nabla V \|_{ L^{5\alpha(p)/(\alpha(p)+3)}(Q)}
&\leq& 
\|U\|_{L^{5\alpha(p)/3}(Q)}\|\nabla V\|_{L^{5}(Q)}
\le C\|U\|_{{X}_{ u}}\|V\|_{{X}_{ v}}.
\label{sss-6}
\end{eqnarray}

On the other hand, 
from the H\"older inequality we have
\begin{eqnarray}
&&
\|( |u^*|^{p-2} u^* - |\theta|^{p-2}\theta )U\|_{L^{5\alpha(p)/(2\alpha(p)+3)}(Q)}
 \nonumber \\
&& \quad \le \|( |u^*|^{p-2} u^* - |\theta|^{p-2}\theta )\|_{L^{5/2}(Q)} 
 \|U\|_{L^{5\alpha(p)/3}(Q)}
\nonumber \\
 &&\quad \le   \,  C \,  \| |u^*|^{p-2} u^* - |\theta|^{p-2}\theta \|_{L^{5/2}(Q)}\|U\|_{X_u}.
\label{sss-6-1}
\end{eqnarray}

Therefore, from \eqref{sss-5}-\eqref{sss-6-1}  
and taking into account \eqref{sss-4-1}, we deduce that
\begin{eqnarray*}
&&\lim_{\| z \|_{{X}_{u}\times {X}_{v}
\times L^{5/2}(L^{5/2+})} \to0}
\frac{\|\mathcal{S}_1(z^*+z)-\mathcal{S}_1(z^*)-\mathcal{S}_1'(z^*)[z]\|_{Y_u}}
{\|z \|_{ {X}_{ u}\times {X}_{v}
\times L^{5/2}(L^{5/2+})}}\\
&&
 \le \lim_{\| z \|_{{X}_{ u}\times{X}_{v}\times L^{5/2}(L^{5/2+})} \to0}
C \left(\| z \|_{{X}_{u}\times{X}_{v}
\!\times L^{5/2}(L^{5/2+})}
+
\!\| |u^*|^{p-2} u^* - |\theta|^{p-2}\theta \|_{L^{ 5/2}(Q)}\right)=0.
\end{eqnarray*}

Consequently, the operator $\mathcal{S}_1$ is Fr\'echet-differentiable and its Fr\'echet derivative is given in
\eqref{sss-2-u}. The continuity of the operator 
$$z^*\in X_u\times X_v\times L^{5/2}(L^{5/2+}(\Omega_c)) \longmapsto S_1'(z^*)[\cdot]\in  \mathcal{L}(X_u\times X_v\times L^{5/2}(L^{5/2+}(\Omega_c));Y_u)$$ can be deduced using the same type of estimates developed previously. In fact, one has

$$
\| \mathcal{S}_1'(z_1^*)-\mathcal{S}_1'(z_2^*)\|_{ \mathcal{L} }
\le C(\| u_1^*-u_2^*\|_{L^{\beta(p)}(Q)} + \| \nabla (v_1^*-v_2^*) \|_{L ^{5}(Q)})
+ \| |u_1^*|^{p-2}u_1^*-|u_2^*|^{p-2}u_2^* \|_{L^{5/2}(Q)}).
$$

Now, we will analyze the Fr\'echet differentiability of $\mathcal{S}_2$. Notice that
$$
\mathcal{S}_2(z^*+z)-\mathcal{S}_2(z^*)
=\partial_t V-\Delta V+ V -
\vert u^*+U \vert^{p} 
+|u^*|^p - (f^* V\,+ F v^* + F V)\,1_{\Omega_c}.
$$
Then, 
from \eqref{sss-2-v} and \eqref{sss-4} we get
\begin{eqnarray*}
&&\|\mathcal{S}_2(z^*+z)-\mathcal{S}_2(z^*)-\mathcal{S}_2'(z^*)[z]\|_{L^{ 5/2}(Q)}\\
&&\ \ \le 
p\, \| ( |u^*|^{p-2} u^* -|\theta|^{p-2}\theta)U\|_{L^{ 5/2}(Q)} + \| F \, V\|_{L^{ 5/2}(Q)}\\
&&\ \ \le  \||u^*|^{p-2} u^* -|\theta|^{p-2}\theta\|_{L^{5\beta(p)/(2\beta(p)-5)}(Q)}
\|U\|_{L^{ \beta(p)}(Q)}\\
&&\ \ \ +   \| F \|_{L^{5/2}(L^{5/2+}(\Omega_c))}\|V\|_{L^\infty(L{^{\infty-}})}\\
&&\ \ \le C  \||u^*|^{p-2} u^* -|\theta|^{p-2}\theta\|_{L^{5\beta(p)/(2\beta(p)-5)}(Q)}\|U\|_{X_u}+ C\| F \|_{L^{5/2}(L^{5/2+}(\Omega_c))}\|V\|_{X_v}.
\end{eqnarray*}


Therefore,  taking into account \eqref{sss-4-1}  and that $\frac{5\beta(p)}{2\beta(p)-5}\le \frac{\beta(p)}{p-1}$
(in fact, $\frac{5\beta(p)}{2\beta(p)-5}=5$ if $p\le2$ and $\frac{5\beta(p)}{2\beta(p)-5}=\frac{5(p-1)}{2p-3}<5$ if $p>2$),
we deduce
\begin{eqnarray*}
&&\lim_{\|z \|_{{X}_{ u}\times{X}_{v}
\times L^{5/2}(L^{5/2+})} \to 0}
\frac{\|\mathcal{S}_2(z^*+z)-\mathcal{S}_2(z^*)-\mathcal{S}_2'(z^*)[z]\|_{L^{ 5/2}(Q)}}
{\|z \|_{{X}_{u}\times{X}_{v} \times L^{5/2}(L^{5/2+})}}\\
&&
\quad \le \lim_{\| z \|_{{X}_{u}\times{X}_{v} \times L^{5/2}(L^{5/2+})} \to0}
C \left( \| |u^*|^{p-2} u^* -|\theta|^{p-2}\theta\|_{L^{5\beta(p)/(2\beta(p)-5)}(Q)} +  \| F \|_{L^{5/2}(L^{5/2+})} \right)=0.
\end{eqnarray*}


Consequently, the operator $\mathcal{S}_2$ is Fr\'echet-differentiable and its Fr\'echet derivative is given in 
\eqref{sss-2-v}. Again, the continuity of the operator 
$$z^*\in X_u\times X_v\times L^{5/2}(L^{5/2+}(\Omega_c)) \longmapsto S_2'(z^*)[\cdot]\in  \mathcal{L}(X_u\times X_v\times L^{5/2}(L^{5/2+}(\Omega_c)); Y_v)
$$
 can be proved. In fact, we get 
\begin{eqnarray*}
\| \mathcal{S}_2'(z_1^*) - \mathcal{S}_2'(z_2^*)\|_{\mathcal{L}} 
 &\le& C( \| |u_1^*|^{p-2}u_1^*-|u_2^*|^{p-2}u_2^* \|_{L ^{5\beta(p)/(2\beta(p)-5)}(Q)} \\
 &&+ \| v_1^*-v_2^* \|_{L^{\infty}(L^{\infty-})} + \| f_1^*-f_2^* \|_{L^{5/2}(L^{5/2+}(\Omega_c))}).
\end{eqnarray*}
\end{proof}
 \subsection{Well-posedness of the linearized system}\label{linealizado}
 \par\noindent
  \begin{proposition}\label{v1}
Let $(u^*,v^*,f^*)\in X_u \times X_v\times L^{5/2}(L^{5/2+}(\Omega_c))$. 
Then, for any $(g_u,g_v)\in Y_u\times Y_v$, 
with $g_u=g^0-\nabla \cdot{\bf g}^1$, there exists a unique solution $(U,V)
\in   {X}_{u}\times {X}_{v}$ of the following linearized system
\begin{equation}
\left\{
\begin{array}{l}
\partial_tU-\Delta U-\nabla\cdot(U\nabla v^*)-\nabla\cdot(u^*\nabla V)
-r\, U+\mu \, p(u^*)^{p-1} U=g^0-\nabla \cdot{\bf g}^1\ \mbox{ in $Q$},\\
\partial_tV-\Delta V-p(u^*)^{p-1}U-f^*V\,1_{\Omega_c}=g_v\ \mbox{ in $Q $},\\
U(0)=V(0)=0\ \mbox{ in }\ \Omega,\\
\left(-\nabla U-U\nabla v^*-u^*\nabla V+{\bf g}^1\right)\cdot{\bf n}=0\ \mbox{ on }\ (0,T)\times\partial\Omega,\\
\partial_{\bf n}V=0\ \mbox{ on }\ (0,T)\times\partial\Omega.\label{eq1e0-1}
\end{array}
\right.
\end{equation}
 \end{proposition}
 
\begin{proof} 
We recall that $Y_v=L^{5/2}(Q)$ and if $g_u\in Y_u$, then $g_u=g^0-\nabla\cdot{\bf g}^1$ with $g^0\in L^{5\alpha(p)/(2\alpha(p)+3)}
(Q)$ and ${\bf g}^1\in L^{5\alpha(p)/(\alpha(p)+3)}(Q)$. 
Since $(u^*,v^*,f^*)\in X_u\times X_v\times L^{5/2}(L^{5/2+}(\Omega_c))$, in particular $(u^*,\nabla v^*)\in L^{\beta(p)}(Q)\times L^5(Q)$. Thus,  by applying Lemma \ref{corol-A} 
for $ a=-r+\mu p(u^*)^{p-1}\in L^{5/2}(Q)$ 
 $c=-\nabla v^*\in L^{5}(Q)$, $d=-u^*\in L^5(Q)$, $\beta_1=-f^*\,1_{\Omega_c}\in L^{5/2}(L^{5/2+})$ and $\beta_2=-p(u^*)^{p-1}\in L^{5}(Q),$ we deduce that there exists a unique $(U,V)\in X_u\times X_v$ solution of \eqref{eq1e0-1}.
\end{proof}
\begin{remark}\label{rem_dual}
As consequence of Proposition \ref{v1} and Lemma \ref{der-oper}, the mapping 
$${\mathcal A}:=\frac{\partial\mathcal{S}(u^*,v^*,f^*)}{\partial(u,v)}: (U,V) \in   {X}_{u}\times {X}_{v} \longmapsto (g_u,g_v)\in Y_u\times Y_v$$
 is linear, bijective  and continuous. Thus, by applying the Banach closed range theorem (see, for instance, \cite[Theorem 3.E]{Zeidler}) it holds that the dual mapping 
 $${\mathcal A}': (\sigma,\eta)\in Y_u'\times Y_v'   \longmapsto     {\mathcal A}'(\sigma,\eta)\in {X}_{u} '\times {X}_{v}' $$ is also linear, bijective and continuous.
\end{remark}

 \subsection{Local differentiability of the control-to-state mapping}

\begin{theorem}\label{v9}

Let $(u^*,v^*,f^*)\in {X}_{u}  \times {X}_{v} \times L^{5/2}(L^{5/2+}(\Omega_c))$ such that $\mathcal{S}(u^*,v^*,f^*)=\bf 0$, 
then the following statements hold:
\begin{enumerate}
\item [i)] There exist neighborhoods $\mathcal{B}_{f ^*}$ of $f^*$ in $ L^{5/2}(L^{5/2+}(\Omega_c))$ 
and $\mathcal{B}_{(u^*,v^*)}$ of $(u^*,v^*)$ in $X_u\times X_v$ such that, for any $f\in\mathcal{B}_{f^*}$,  
system (\ref{eq1})-(\ref{eq2}) or (\ref{eq1b})-(\ref{eq2}) with control $f$ has a unique  solution $(u_{f},v_{f})\in \mathcal{B}_{(u^*,v^*)}.$ 
\item  [ii)] The map $\mathcal{E}:\mathcal{B}_{f^*} \rightarrow \mathcal{B}_{(u^*,v^*)}$ given by $\mathcal{E}(f)=(u_{f},v_{f})$ is of class $C^1.$ 
\item  [iii)] If $(\overline{U},\overline{V}):=D\mathcal{E}(f)[F],$ for some $f\in \mathcal{B}_{f^*}$ and $F\in L^{5/2}(L^{5/2+}(\Omega_c))$,
 then $(\overline{U},\overline{V})$ 
  is the unique  solution in $\mathcal{B}_{(u^*,v^*)}$ of the ``linearized'' problem 
\begin{align}
	\left\{
	\begin{array}
		[c]{lll}%
		\partial_t \overline{U}-\Delta \overline{V} -\nabla\cdot (\overline{U}\nabla v_{f})
		- \nabla\cdot(u_{f}\nabla \overline{V})
		-r \overline{U} + \mu p(u_{f})^{p-1} \overline{U}=0\ \mbox{ in $Q$,} \\
		\partial_t \overline{V} -\Delta \overline{V} + \overline{V} - p (u_{f})^{p-1} {\overline{U}} - f\overline{V}1_{\Omega_c}
		= F v_f 1_{\Omega_c}\ \mbox{ in $Q$,}\\
		\overline{U}(0)=0, \ \overline{V}(0)=0\ \mbox{ in $\Omega$,}\\	
		\partial_{\bf n} \overline{U}=0,\ \partial_{\bf n} \overline{V}=0\ \mbox{ on $(0,T)\times\partial\Omega$.}
	\end{array}
	\right.  \label{eq1eb}%
\end{align}
\end{enumerate}
\end{theorem}
\begin{proof}
From Lemma \ref{der-oper} we have that equality operator $\mathcal{S}\in C^1$, and 
\begin{eqnarray*}
\frac{\partial\mathcal{S}(u^*,v^*,f^*)}{\partial(u,v)}(\overline{U},\overline{V})=
\begin{pmatrix}
\partial_t \overline{U}-\Delta \overline{U} -\nabla\cdot (\overline{U}\nabla v^*)- \nabla\cdot(u^*\nabla \overline{V})-r \overline{U}+\mu p(u^*)^{p-1}\overline{U} \\
\partial_t \overline{V} -\Delta \overline{V} + \overline{V} -p (u^*)^{p-1} {\overline{U}}-f^*\overline{V}1_{\Omega_c}
\end{pmatrix}.
\end{eqnarray*}
Now, applying Proposition \ref{v1},  we get that
  $\frac{\partial\mathcal{S}(u^*,v^*,f^*)}{\partial(u,v)}$ is an isomorphism from 
 $X_{ u} \times X_{v} $ onto 
 $ Y_u\times Y_v$.
 Since $(u_{f^*},v_{f^*})$ is the  solution of (\ref{eq1})-(\ref{eq2}) or (\ref{eq1b})-(\ref{eq2}) with data $f^*$,  
 in particular $\mathcal{S}(u_{f^*},v_{f^*},f^*)=(0,0)$.
 Thus, the Implicit Function Theorem in Banach spaces (see, for instance, \cite[Theorem 3.13]{Kumar}) implies the existence of open neighborhoods 
 $\mathcal{B}_{f^*}$ of $f^*$ in $L^{5/2}(L^{5/2+}(\Omega_c))$ and $\mathcal{B}_{(u^*,v^*)}$ of $(u^*,v^*)$ in $X_u\times X_v$,  and a $C^1$ control-to-state map 
 $\mathcal{E}:\mathcal{B}_{f^*} \rightarrow \mathcal{B}_{(u^*,v^*)}$ such that $\mathcal{S}(\mathcal{E}(f),f)=(0,0),$ for each $f\in \mathcal{B}_{f^*}$. 
Then, $\mathcal{E}(f)=(u_{f},v_{f})$ is the  solution (\ref{eq1})-(\ref{eq2}) or (\ref{eq1b})-(\ref{eq2}) with control $f$ and initial data $(u_0,v_0).$ In addition,
\begin{eqnarray*}
\frac{\partial\mathcal{S}(u_f,v_f,f)}{\partial(u,v)} (D\mathcal{E}(f)[F])
+
\frac{\partial\mathcal{S}(u_f,v_f,f)}{\partial f}[F]=(0,0),\quad  \forall\, F\in  L^{5/2}(L^{5/2+}(\Omega_c)).
\end{eqnarray*}
Considering $(\overline{U},\overline{V})=D\mathcal{E}(f)[F]$, and accounting that 
$$  
\frac{\partial\mathcal{S}(u_f,v_f,f)}{\partial f}[F]=(0, -F v_f 1_{\Omega_c}),
$$
we deduce that $(\overline{U},\overline{V})$ satisfies (\ref{eq1eb}).
 \end{proof}
 
 \subsection{Differentiability of the functional cost}

\begin{lemma}\label{DerJ} The functional $J:{X}_{u}  \times {X}_{v} \times L^{5/2}(L^{5/2+}(\Omega_c))\rightarrow \mathbb{R}$ is Fr\'echet differentiable and
its derivative in $z=(u,v,f)\in {X}_{u}  \times {X}_{v} \times L^{5/2}(L^{5/2+}(\Omega_c)),$ in the direction 
$\overline z=(U,V,F)\in {X}_{u}  \times {X}_{v} \times L^{5/2}(L^{5/2+}(\Omega_c))$ is 
 \begin{eqnarray}\label{derJ}
J'(z)[\overline z]&=&\gamma_u\left(\int_0^T\int_\Omega {\rm sgn}(u-u_d)\vert u-u_d\vert^{(5p-2)/2}U
+\int_0^T\int_\Omega {\rm sgn}(u-u_d)\vert u-u_d\vert^{7/3}U\right)
\nonumber\\
&&+\gamma_v\int_0^T\int_\Omega (v-v_d)V
+\gamma_f\int_0^T\int_{\Omega_c}\frac{{\rm sgn}(f)\vert f\vert^{3/2+}}{\|f\|^{0+}_{L^{5/2+}(\Omega_c)}}F.
\end{eqnarray}
\end{lemma}
\begin{proof}The cost functional is sum of real-valued functionals of $L^p$-$L^q$-norms of the form $\Psi_{p,q}(y)=\frac{1}{p}\Vert y\Vert^p_{L^p(L^q)}$, with $1<p,q<\infty$. Following \cite[Theorem 2.5]{leonard_1974} (see also, \cite[Corollary A.4]{Kunisch}), its Fr\'echet derivative is given by 
$$
\Psi'_{p,q}(y)[h]
=\int_0^T
\left(\int_\Omega \Vert y(t,\cdot)\Vert^{p-q}_{L^q} \vert y(t,x)\vert^{q-2}y(t,x)
\,h(t,x)\right).
$$
Recalling the definition of the sign function and applying the above representation, we get directly the expression (\ref{derJ}).
 \end{proof}

\subsection{Well-posedness and regularity of the adjoint problem}

\

Let $(u^*,v^*,f^*)\in\mathcal{S}_{ad}$ be any local optimal solution of problem \eqref{C-3}. By using the adjoint operator ${\mathcal A}'$ defined in Remark \ref{rem_dual}, we can rewrite the  adjoint system \eqref{adjz0} as 
$$
{\mathcal A}'(\sigma,\eta)=
(\gamma_u h(u^*-u_d),\gamma_v (v^*-v_d))
$$ 
where $h(w):=\left(|w|^{(5p-4)/2}+|w|^{4/3}\right)w$. Then, since ${\mathcal A}'$ is bijective (Remark \ref{rem_dual}) it suffices to prove $h(u^*-u_d)\in X_u'$ and $v^*-v_d\in X_v'$. Indeed,  
since $u_d\in L^{5p/2}(Q)\cap L^{10/3}(Q)$ then $|u^*-u_d|^{(5p-2)/2}\in L^{5p/(5p-2)}(Q)=L^{5p/2}(Q)'$ and $|u^*-u_d|^{7/3}\in L^{10/7}(Q)=L^{10/3}(Q)'$. Therefore, since $X_u\subset L^{\beta(p)}(Q)\subset L^{5p/2}(Q)\cap L^{10/3}(Q)$, then $h(u^*-u_d)\in X_u'$. On the other hand, since $v_d\in L^{2}(Q)$ and $X_v\subset L^2(Q)$,  then 
 $v^*-v_d\in X_v'$.
 
Therefore, by the previous duality argument and  taking into account that
 $Y_u'=Z''=Z$ (because $Z$ is a reflexive Banach space) and $Y_v'=L^{5/3}(Q)$,  
  the following result holds.
 
\begin{proposition}
There exists a unique  solution   $(\sigma,\eta)\in Z \times L^{5/3}(Q)$ of the adjoint problem (\ref{adjz0}).
In particular, $\sigma\in L^{5\alpha(p)/(3\alpha(p)-3)}(Q)$ with $\nabla\sigma\in L^{5\alpha(p)/(4\alpha(p)-3)}(Q)$.
\end{proposition}

Moreover, we can obtain more regularity of  (\ref{adjz0}) whenever $p\le 4/3$.

\begin{proposition} \label{Reg-mult}
For $p\le 4/3$, the  solution of (\ref{adjz0})  satisfies $(\sigma, \eta)\in W_2 \times  L^2(Q)$. 
\end{proposition}
\begin{proof}
 Since $u_d\in L^{5p/2}(Q)\cap L^{10/3}(Q)$, then $h({u^*-u_d})\in L^{5p/(5p-2)}(Q) + L^{10/7}(Q)$. In particular $\gamma_u h({u^*-u_d})\in L^{10/7}(Q)$ whenever $p\le 4/3$. On the other hand,  
$\gamma_v({v^*}-v_d)\in L^2(Q)$ and $(u^*)^{p-1}\in L^{5/2}(Q)$. Therefore, we can use  \cite[Theorem A.1, case 2b)]{Andre_SICOM} to deduce that $({\sigma},{\eta})\in W_2 \times  L^2(Q) $.
\end{proof}

%
%

\subsection{
Proof of Theorem \ref{FOC}.} \label{Sec:FOC}

\

Notice that $(u^*,v^*,f^*)\in X_u\times X_v\times L^{5/2}(L^{5/2+}(\Omega_c))$ such that $\mathcal{S}(u^*,v^*,f^*)=\bf 0$, then 
 from Theorem \ref{v9}, the control-to-state mapping is well defined locally near of $(u^*,v^*,f^*)$ in the $X_u\times X_v\times L^{5/2}(L^{5/2+}(\Omega_c))$-topology. Thus, we can define the reduced cost functional $J_0:\mathcal{B}_{f^*}\rightarrow \mathbb{R}$ by $J_0(f):=J(u_f,v_f,f)$, for any $f$ near of $f^*$ in the $L^{5/2}(L^{5/2+}(\Omega_c))$-norm and the corresponding $(u_f,v_f)$ near of $(u^*,v^*)$ in the $X_u\times X_v$-norm. 
 
We compute the derivative $J'_0(f^*)[f]$ for any  $f\in \mathcal{F}.$ Let $(U,V):=D\mathcal{E}(f^*)[f].$ Then, from Theorem \ref{v9}, the pair $(U,V)\in X_u\times X_v$ is the unique solution of
\begin{align} \label{eq1ec}
	\left\{
	\begin{array}
		[c]{lll}%
		\partial_t U-\Delta U -\nabla\cdot (U\nabla v^*)- \nabla\cdot(u^*\nabla V)-r U+\mu p(u^*)^{p-1}U=0\  \mbox{ in}\  Y_u=Z', \\
		 \noalign{\vspace{-2ex}}\\
		\partial_t V -\Delta V + V - p (u^*)^{p-1} U-f^*V1_{\Omega_c}=fv^*1_{\Omega_c}\  a.e.\ (t,x)\in Q,\\
		 \noalign{\vspace{-2ex}}\\
		U(0)=0, \quad V(0)=0\quad \mbox{in}\ \Omega,\\	
		 \noalign{\vspace{-2ex}}\\
		\partial_{\bf n} U=0,\quad \partial_{\bf n} V=0\quad \mbox{on }(0,T)\times\partial\Omega.
	\end{array}
	\right.  %
\end{align}
From the chain rule in Banach spaces (cf. \cite[Theorem 3.6]{Kumar}) and Lemma \ref{DerJ} we obtain
\begin{eqnarray}
&&J_0'(f^*)[f]=\frac{\partial J(u^*,v^*,f^*)}{\partial (u,v)}
( D\mathcal{E}(f^*)[f])+\frac{\partial J(u^*,v^*,f^*)}{\partial f}[f]\nonumber\\
&&\ \ = \frac{\partial J(u^*,v^*,f^*)}{\partial (u,v)}(U,V)
+\frac{\partial J(u^*,v^*,f^*)}{\partial f}[f]\nonumber\\
&&\ \ =\!\gamma_u\int_0^T\!\int_\Omega 
h({u^*-u_d}) 
U+\!\gamma_v\int_0^T\!\int_\Omega (v^*-v_d)V+\!\gamma_f\int_0^T\!\int_{\Omega_c} 
\frac{{\rm sgn}(f^*)\vert f^*\vert^{3/2+}}{\|f^*\|^{0+}_{L^{5/2+}(\Omega_c)}}{f}.
\label{chain2}
\end{eqnarray} 
 In this point we consider the adjoint system (\ref{adjz0}). Then, testing (\ref{adjz0})$_1$ by $U$ and (\ref{adjz0})$_2$ by $V,$  integrating by parts and using \eqref{eq1ec} we get
\begin{eqnarray}\label{gr1}
\gamma_u\int_0^T\int_\Omega 
h({u^*-u_d}) 
U+\gamma_v\int_0^T\int_\Omega (v^*-v_d)V
=\int_0^T\int_{\Omega_c} {v^*}\eta f.
\end{eqnarray}
Therefore, replacing (\ref{gr1}) into (\ref{chain2}) we conclude that  
\begin{eqnarray}\label{cop2}
J_0'(f^*)[f]=\int_0^T\int_{\Omega_c}\left(
\gamma_f 
\frac{{\rm sgn}(f^*)\vert f^*\vert^{3/2+}}{\|f^*\|^{0+}_{L^{5/2+}(\Omega_c)}}+{v^*}\eta\right)f , 
\quad  \forall {f}\in\mathcal{F}.
\end{eqnarray}
In particular, identifying $(L^{5/2}(L^{5/2+}))' \equiv L^{5/3}(L^{5/3-})$, we have the identification
$$
J_0'(f^*) \equiv \gamma_f
\frac{{\rm sgn}(f^*)\vert f^*\vert^{3/2+}}{\|f^*\|^{0+}_{L^{5/2+}(\Omega_c)}}
+{v^*}\eta.
$$
\subsection{
Proof of Corollary~\ref{optim-cond}.} \label{Sec:optim-cond}

\

Now, we are in position to derive the first-order optimality conditions and establish the optimality system. 
Let $f^*$ be any local optimal control of Problem (\ref{C-3}), then $f^*$ is a solution of the reduced (local) optimal control problem 
\begin{equation}\label{C-3b}
\left\{
\begin{array}{lc}
\min J_0(f),\\
\mbox{subject to}\ f\in \mathcal{B}_{f^*}\cap \mathcal{F}.
\end{array}
\right.
\end{equation}
Taking into account that $\mathcal{F}$ is convex and $\mathcal{B}_{f^*}$ is an open  neighbourhood of $f^*,$ for each $f\in \mathcal{F}$ there exists $\epsilon_f>0,$ depending of $f,$ such that $f^\epsilon=f^*+\epsilon(f-f^*)\in \mathcal{B}_{f^*}\cap\mathcal{F},$ for all $\epsilon\in [0,\epsilon_f]$.
In addition,  since $J_0\in C^1$ in $f^*$  and $\mathcal{F}$ is a convex subset of $L^{5/2}(L^{5/2+}(\Omega_c))$ we get
\begin{eqnarray}\label{cop1}
J'_0(f^*)[f-f^*]
=\lim_{\epsilon\rightarrow 0}\frac{J_0(f^\epsilon)-J_0(f^*)}{\epsilon}\geq 0\quad \forall f\in \mathcal{F}.
\end{eqnarray}
Finally, from (\ref{cop2}) along with (\ref{cop1}) implies the variational inequality (\ref{L-3}).

\bigskip

\appendix

\section{Existence for linear diffusion-convection-reaction systems}\label{AppendixA}

Extending \cite[Lemma 3.1]{Andrei_AMO} by considering additional divergence terms, we can state the following result:

 \begin{lemma}\label{lem-B}
 Let 
 $2\le q<\infty$ and 
 $\Omega\subset\mathbb{R}^3$ be a bounded domain such that $\partial\Omega\in C^2$. If 
 $a\in L^{5/2}(Q)$, ${\bf c}\in (L^5(Q))^3$, 
 $g^0\in L^{5q/(2q+3)}(Q)$, 
 ${\bf g}^1\in (L^{5q/(q+3)}(Q))^3$ 
 and 
 $w_0\in L^q(\Omega)$,  then the linear diffusion-convection-reaction problem 
 \begin{equation}\label{lem-B-1}
 \left\{
 \begin{array}{rcl}
 \partial_tw-\Delta w+\nabla\cdot(w\, {\bf c})+aw&=&g^0+\nabla\cdot{\bf g}^1\ \mbox{ in }\ Q,\\
 w(0)&=&w_0\ \mbox{ in }\ \Omega,\\
 (-\nabla w+w\, {\bf c}-{\bf g}^1)\cdot{\bf n}&=&0\ \mbox{ on }\ (0,T)\times\partial\Omega,
 \end{array}
 \right.
 \end{equation}
 has a unique weak solution 
 $w$ such that $|w|^{q/2}\in W_2$. Moreover, 
 $w\in L^\infty(L^q)\cap L^{5q/3}(Q)$ 
  and $\nabla w\in L^2(Q)$. 
  \end{lemma}
 
 \begin{proof} 
 Let $a_{\varepsilon},{\bf c}_\varepsilon, g^0_\varepsilon, {\bf g}_\varepsilon^1$ and $w^0_\varepsilon$ adequate regularizations such that 
 $a_{\varepsilon}\to a$ in $L^{5/2}(Q)$, $ {\bf c}_\varepsilon\to {\bf c}$ in $L^5(Q)^3$, $g^0_\varepsilon\to g^0$ in $L^{5q/(2q+3)}(Q)$, 
 ${\bf g}_\varepsilon^1\to {\bf g}^1$ in $L^{5q/(q+3)}(Q)^3$ with $\nabla\cdot{\bf g}^1_\varepsilon\in L^{5/2}(Q)$ 
  and $w_\varepsilon^0\to w_0$ in $L^{q}(\Omega)$. Then, we first solve the following regularized problem: Find $w_\varepsilon$ such that
 \begin{equation}\label{lem-B-2}
 \left\{
 \begin{array}{rcl}
 \partial_tw_\varepsilon-\Delta w_\varepsilon+\nabla\cdot(w_\varepsilon{\bf c}_\varepsilon)+a_{\varepsilon} w_\varepsilon&=&g^0_\varepsilon+\nabla\cdot{\bf g}^1_\varepsilon\ \mbox{ in $Q$},\\
 w_\varepsilon(0)&=&w_\varepsilon^0\ \mbox{ in $\Omega$},\\
 (-\nabla w_\varepsilon+w_\varepsilon{\bf c}_\varepsilon-{\bf g}_\varepsilon^1)\cdot{\bf n}&=&0\ \mbox{ on $(0,T)\times\partial\Omega$}.
 \end{array}
 \right.
 \end{equation}
 To solve \eqref{lem-B-2} we use the Leray-Schauder fixed-point theorem, for which we define the operator 
 $$
 \Gamma_\varepsilon: W_{5/2}:= L^{5/2}(W^{1,5/2})\to X_{5/2}\hookrightarrow W_{5/2}
 $$
 by $\Gamma_\varepsilon(\overline{w}_\varepsilon)=w_\varepsilon$, where $w_\varepsilon$ is the solution of problem 
 \begin{equation}\label{lem-B-3}
 \left\{
 \begin{array}{rcl}
 \partial_tw_\varepsilon-\Delta w_\varepsilon&=&-\nabla\cdot(\overline{w}_\varepsilon{\bf c}_\varepsilon)-a_{\varepsilon} \overline{w}_\varepsilon +g^0_\varepsilon+\nabla\cdot{\bf g}^1_\varepsilon\ \mbox{ in }\ Q,\\
 w_\varepsilon(0)&=&w_\varepsilon^0\ \mbox{ in }\ \Omega,\\
 (-\nabla w_\varepsilon+\overline w_\varepsilon{\bf c}_\varepsilon-{\bf g}_\varepsilon^1)\cdot{\bf n}&=&0\ \mbox{ on }\ (0,T)\times\partial\Omega.
 \end{array}
 \right.
 \end{equation}
 From the regularity of $\overline{w}_\varepsilon,{\bf c}_\varepsilon,a_\varepsilon,g^0_\varepsilon$ and ${\bf g}^1_\varepsilon$ we have in particular that $-\nabla\cdot(\overline{w}_\varepsilon{\bf c}_\varepsilon)-a_\varepsilon \overline{w}_\varepsilon +g^0_\varepsilon+\nabla\cdot{\bf g}^1_\varepsilon$ belongs to $L^{5/2}(Q)$. Then, the parabolic regularity gives the existence of a unique solution $w_\varepsilon\in X_{5/2}$ of \eqref{lem-B-3}. Thus, the mapping $\Gamma_\varepsilon$ is well-defined. The continuity of $\Gamma_\varepsilon$ from $W_{5/2}$ into itself is clear. Moreover, from Simon-Aubin-Lions compactness,  the embedding $X_{5/2}\hookrightarrow W_{5/2}$ is compact; thus the operator $\Gamma_\varepsilon$ is compact from $W_{5/2}$ into $X_{5/2}$.  Therefore, to finish the existence of solutions of system \eqref{lem-B-2} as a fixed-point of $\Gamma_\varepsilon,$ it suffices to prove that the set of possible fixed-points of $\alpha\Gamma_\varepsilon$, for any $\alpha\in[0,1]$, is bounded in $W_{5/2}$. Indeed, let $\alpha\in(0,1]$ (the case $\alpha=0$ is clear). Then, if $w_\varepsilon\in X_{5/2}$ is fixed point of $\alpha\Gamma_\varepsilon$ satisfies the following problem:
 
  \begin{equation}\label{lem-B-4}
 \left\{
 \begin{array}{rcl}
 \partial_t w_\varepsilon-\Delta w_\varepsilon&=&-\alpha\nabla\cdot({w}_\varepsilon{\bf c}_\varepsilon)-\alpha a_{\varepsilon}{w}_\varepsilon +\alpha g^0_\varepsilon+\alpha\nabla\cdot{\bf g}^1_\varepsilon\ \mbox{ in }\ Q,\\
 w_\varepsilon(0)&=&w_\varepsilon^0\ \mbox{ in }\ \Omega,\\
(-\nabla w_\varepsilon+\alpha w_\varepsilon{\bf c}_\varepsilon- \alpha {\bf g}_\varepsilon^1)\cdot{\bf n}&=&0\ \mbox{ on }\ (0,T)\times\partial\Omega.
 \end{array}
 \right.
 \end{equation}
 Since the right hand side of \eqref{lem-B-4} is bounded in $L^{5/2}(Q)$ independent of $\alpha$ (although depending on  $\varepsilon$), there exists a unique $w_\varepsilon$ solution of problem  \eqref{lem-B-2}.
 
 Now, we are going to make estimates of $w_\varepsilon$ independent of $\varepsilon$. 
 Since $w_\varepsilon\in X_{5/2}$, then $w_\varepsilon \in L^\infty(L^{\infty-})$. 
 Thus, taking $|w_\varepsilon|^{q-2}w_\varepsilon$ (for any $q\ge 2$) as test function in \eqref{lem-B-4} (and using the equality $\nabla ( |w_\varepsilon|^{q-2}w_\varepsilon) = {\rm sgn} (w_\varepsilon)\nabla (|w_\varepsilon|^{q-1})$), 
 we obtain
 \begin{eqnarray}
 &&\frac1q\frac{d}{dt}\||w_\varepsilon|^{q/2}\|^2+\frac{4(q-1)}{q^2}\|\nabla(|w_\varepsilon|^{q/2})\|^2
 = -\alpha\int_\Omega a_{\varepsilon} |w_\varepsilon|^q
 +\alpha\int_\Omega |w_\varepsilon| {\bf c}_\varepsilon \cdot \nabla (|w_\varepsilon|^{q-1})\nonumber\\
 && \quad+ \alpha\int_\Omega {\rm sgn} (w_\varepsilon)\, {\bf g}^1_\varepsilon\cdot\nabla(|w_\varepsilon|^{q-1})+\alpha\int_\Omega g^0_\varepsilon|w_\varepsilon|^{q-2}w_\varepsilon.\label{lem-B-5}
 \end{eqnarray}
 Now we will bound the terms on the right-hand side of \eqref{lem-B-5}. Applying the H\"older and Young inequalities and taking into account the estimate $\|v\|_{L^{10/3}}\le C\|v\|^{2/5}\|v\|^{3/5}_{H^1}$, we can obtain
 
\begin{eqnarray}
-\alpha\int_\Omega a_{\varepsilon} |w_\varepsilon|^q
\le\alpha\|a_{\varepsilon}\|_{L^{5/2}}\||w_\varepsilon|^{q/2}\|_{L^{10/3}}^2 
\le \delta\||w_\varepsilon|^{q/2}\|^2_{H^1}+C\|a_{\varepsilon}\|_{L^{5/2}}^{5/2}\||w_\varepsilon|^{q/2}\|^2,\label{lem-B-6}
\end{eqnarray}
\begin{eqnarray}
\alpha\int_\Omega |w_\varepsilon| {\bf c}_\varepsilon \cdot \nabla (|w_\varepsilon|^{q-1})
&\le&\frac2q\alpha\int_\Omega |w_\varepsilon|^{q/2} |\nabla(|w_\varepsilon|^{q/2})|\, |{\bf c}_\varepsilon|\nonumber\\
&\le&C\||w_\varepsilon|^{q/2}\|_{L^{10/3}}
\||w_\varepsilon|^{q/2}\|_{H^1}\|{\bf c}_\varepsilon\|_{L^5}\nonumber\\
&\le&\delta\||w_\varepsilon|^{q/2}\|^2_{H^1}+C\||w_\varepsilon|^{q/2}\|^2\|{\bf c}_\varepsilon\|^5_{L^5},\label{lem-B-7}
\end{eqnarray}
\begin{eqnarray}
\alpha\int_\Omega {\rm sgn} (w_\varepsilon) {\bf g}^1_\varepsilon\cdot\nabla(|w_\varepsilon|^{q-1})
&\le&\alpha\int_\Omega |{\bf g}^1_\varepsilon | (|w_\varepsilon|^{q/2})^{\frac{q-2}q} |\nabla(|w_\varepsilon|^{q/2})|
\nonumber\\
&\le&\alpha\|{\bf g}_\varepsilon^1\|_{L^{\frac{5q}{q+3}}}\||w_\varepsilon|^{q/2}\|_{L^{10/3}}^{\frac{q-2}q}\|\nabla(|w_\varepsilon|^{q/2})\|\nonumber\\
&\le&C\|{\bf g}^1_\varepsilon\|_{L^{\frac{5q}{q+3}}}\||w_\varepsilon|^{q/2}\|^{\frac{2(q-2)}{5q}}\||w_\varepsilon|^{q/2}\|_{H^1}^{\frac{8q-6}{5q}}\nonumber\\
&\le&\delta\||w_\varepsilon|^{q/2}\|^2_{H^1}+C\|{\bf g}^1_\varepsilon\|_{L^{\frac{5q}{q+3}}}^{\frac{5q}{q+3}}
\||w_\varepsilon|^{q/2}\|^{\frac{2(q-2)}{q+3}} 
\nonumber \\
&\le&
\delta\||w_\varepsilon|^{q/2}\|^2_{H^1}
+
\tilde{C}
\|{\bf g}^1_\varepsilon\|_{L^{\frac{5q}{q+3}}}^{\frac{5q}{q+3}}
\left(\||w_\varepsilon|^{q/2}\|^2 + 1 \right).\label{lem-B-8}
\end{eqnarray}
Note that in previous estimate $q\ge 2$ is used. Finally,
\begin{eqnarray}
\alpha\int_\Omega g^0_\varepsilon|w_\varepsilon|^{q-2}w_\varepsilon
&\le&\alpha \|g_\varepsilon^0\|_{L^{\frac{5q}{2q+3}}}\||w_\varepsilon|^{q/2}\|_{L^{10/3}}^{2(q-1)/q}\nonumber\\
&\le&C\|g_\varepsilon^0\|_{L^{\frac{5q}{2q+3}}}\||w_\varepsilon|^{q/2}\|^{\frac{4(q-1)}{5q}}\||w_\varepsilon|^{q/2}\|^{\frac{6(q-1)}{5q}}_{H^1}\nonumber\\
&\le&\delta\||w_\varepsilon|^{q/2}\|^2_{H^1}+C\|g^0_\varepsilon\|_{L^{\frac{5q}{2q+3}}}^{\frac{5q}{2q+3}}\||w_\varepsilon|^{q/2}\|^{\frac{4(q-1)}{2q+3}}\nonumber \\
&\le&
\delta\||w_\varepsilon|^{q/2}\|^2_{H^1}
+
\tilde{C}
\|g^0_\varepsilon\|_{L^{\frac{5q}{2q+3}}}^{\frac{5q}{2q+3}}
\left(\||w_\varepsilon|^{q/2}\|^2 + 1 \right)
.\label{lem-B-9}
\end{eqnarray}
Hence, carrying estimates \eqref{lem-B-6}-\eqref{lem-B-9} into \eqref{lem-B-5} and adding $\||w_\varepsilon|^{q/2}\|^2$ to both sides of the resulting inequality, we arrive at 
\begin{eqnarray}
\frac{d}{dt}\||w_\varepsilon|^{q/2}\|^2+C\||w_\varepsilon|^{q/2}\|^2_{H^1}
&\le&C\left(\|a_{\varepsilon}\|^{5/2}_{L^{5/2}}+\|{\bf c}_\varepsilon\|_{L^5}^5
+
\|{\bf g}_\varepsilon^1\|_{L^{\frac{5q}{q+3}}}^{\frac{5q}{q+3}}+\|g^0_\varepsilon\|_{L^{\frac{5q}{2q+3}}}^{\frac{5q}{2q+3}}
\right)\||w_\varepsilon|^{q/2}\|^2
\nonumber\\
&&+C\left(\|{\bf g}_\varepsilon^1\|_{L^{\frac{5q}{q+3}}}^{\frac{5q}{q+3}}+\|g^0_\varepsilon\|_{L^{\frac{5q}{2q+3}}}^{\frac{5q}{2q+3}}\right).\label{lem-B-10}
\end{eqnarray}
Therefore, from \eqref{lem-B-10} and Gronwall lemma we deduce that 
$\{|w_\varepsilon|^{q/2}\}_{\varepsilon>0}$ is bounded in $W_2$ (for any $q\ge 2$), independently of $\varepsilon>0$. 
Passing to the limit, as $\varepsilon$ goes to 0, we obtain that $w_\varepsilon$ converges to a solution $w$ of problem \eqref{lem-B-1} such that $|w|^{q/2}\in W_2$. Consequently $w\in L^\infty(L^q)\cap L^{5q/3}(Q)$. Since $q\ge 2$, the previous estimate holds also for $q=2$, hence in particular $\nabla w\in L^2(Q)$.
%

Finally, the uniqueness of solutions can be obtained by following a classical comparison argument, because looking at the regularity of each term of equation \eqref{lem-B-1}$_1$ one can take the weak solution as test function.
 \end{proof}
 
  \begin{corollary}\label{lem-A}
 Let $1< q<\infty$ and 
 $\Omega\subset\mathbb{R}^3$ be a bounded domain with $\partial\Omega\in C^2$. 
 Assume that $w_0\in L^{q}(\Omega)$,  $a\in L^{5/2}(Q)$ and $g\in L^{5q/(2q+3)}(Q)$. Then the
 linear diffusion-reaction problem 
 \begin{equation}\label{lem-A-1}
 \left\{
 \begin{array}{rcl}
 \partial_tw-\Delta w+a\, w&=&g\ \mbox{ in }\ Q,\\
 w(0)&=&w_0\ \mbox{ in }\ \Omega,\\
 \partial_{\bf n}w&=&0\ \mbox{ on }\ (0,T)\times\partial\Omega,
 \end{array}
 \right.
 \end{equation}
 has a unique weak solution 
 $w$ such that $|w|^{q/2}\in W_2$. Moreover, 
 $w\in L^\infty(L^q)\cap L^{5q/3}(Q)$ with $\nabla w\in L^{5q/(q+3)}(Q)$ for $q\in(1,2]$, and $\nabla w\in L^2(Q)$ for $q\ge2$.  In particular, if $w_0\in L^{\infty-}(\Omega)$ and $g\in L^{5/2-}(Q)$, then  
 $w\in L^\infty(L^{\infty-})$ with $\nabla w\in L^2(Q)$. 
  Moreover, if $a\in L^{5/2}(L^{5/2+})$, $g\in L^{5/2}(Q)$ and $w_0\in W^{6/5,5/2}(\Omega)$, 
  then $w\in X_{5/2}$.
 
 \end{corollary}
 \begin{proof} 
 This result follows the same argument of Lemma \ref{lem-B}, excepting uniqueness when $1<q<2$ and the regularity  $w\in X_{5/2}$. 
 In fact, we will prove uniqueness of weak solutions $w\in L^{5q/3}(Q)$ of system \eqref{lem-A-1} arguing by duality. Indeed, since $a\in L^{5/2}(Q)$ 
  then the adjoint problem
$$
-\partial_t\varphi-\Delta\varphi
+a\varphi=g_\varphi,\ \ \varphi(T,\cdot)=0,
$$
endowed with the corresponding isolated boundary conditions, has a unique strong solution $\varphi\in X_{5q/(5q-3)}$ for any data $g_\varphi\in L^{5q/(5q-3)}(Q)$ (see \cite[Theorem A.1]{Andre_SICOM}). 
Then, we consider $w_1,w_2\in L^{5q/3}(Q)$ two possible weak solutions of \eqref{lem-A-1}, thus testing by any function $\varphi\in 
X_{5q/(5q-3)}$ such that $\partial_{\bf n}\varphi=0$ on $(0,T)\times \partial\Omega$, we have
$$
0=\int_0^T\int_\Omega(w_1-w_2)\left(-\partial_t\varphi-\Delta\varphi
+a\varphi\right)
=\int_0^T\int_\Omega(w_1-w_2)g_\varphi.
$$
Then, by density $w_1=w_2$ in $L^{5q/3}(Q)$, which implies the uniqueness. 

On the other hand, since $w(0)=w_0\in W^{6/5,5/2}(\Omega)$ and $g-aw\in L^{5/2}(Q)$, from parabolic regularity we deduce that $w\in X_{5/2}$.
\end{proof}

 \begin{lemma}\label{lem-C}
 Let $\Omega\subset\mathbb{R}^3$ be a bounded domain with boundary $\partial\Omega\in C^2$. Suppose that $a\in L^{5/2}(L^{5/2+})$, $g\in L^{5/2}(Q)$ and $w_0\in W^{1,3}(\Omega)$, then the linear 
 diffusion-reaction problem
 \begin{equation}\label{lem-C-1}
 \left\{
 \begin{array}{rcl}
 \partial_tw-\Delta w+a\, w&=&g\ \mbox{ in }\ Q,\\
 w(0)&=&w_0\ \mbox{ in }\ \Omega,\\
 \partial_{\bf n}w&=&0\ \mbox{ on }\ (0,T)\times\partial\Omega,
 \end{array}
 \right.
 \end{equation}
 has a unique weak solution $w$ such that $|w|^{3/2}$ and 
 $|\nabla w|^{3/2}$ belong to $W_2$. 
 In particular  $w\in L^\infty(W^{1,3})$.
 \end{lemma}
 
 \begin{proof} 
 We consider a regularization $w_\varepsilon^0\in W^{6/5,5/2}(\Omega)$ such that $w_\varepsilon^0\to w_0$ in $W^{1,3}(\Omega)$. 
 Then,  following Corollary \ref{lem-A}  
 we can consider the solution $w_\varepsilon\in X_{5/2}$ of the problem 
  \begin{equation}\label{lem-C-2}
 \partial_tw_\varepsilon-\Delta w_\varepsilon+aw_\varepsilon=g,\quad w_\varepsilon(0,\cdot)=w^0_\varepsilon,\quad \partial_{\bf n}w_\varepsilon|_{\partial\Omega}=0.
 \end{equation}
 Since $w_\varepsilon \in X_{5/2}$, then $w_\varepsilon\in L^\infty(L^{\infty-})$ and $\nabla w_\varepsilon\in L^5(Q)$. Thus, we can take $|w_\varepsilon|w_\varepsilon
 -\nabla\cdot (|\nabla w_\varepsilon|\nabla w_\varepsilon)$ as test function in \eqref{lem-C-2}. In order to simplify the argument to obtain a priori estimates, we will divide it into three steps:
 \vspace{0.1cm}
 
 \noindent\underline{\it Step 1:} Estimates for $|w_\varepsilon|^{3/2}$.

By testing \eqref{lem-C-2} by $|w_\varepsilon|w_\varepsilon$ and arguing as in  \eqref{lem-B-6} and \eqref{lem-B-9}  for $q=3$ we can obtain
\begin{equation}\label{lem-C-3}
\frac{d}{dt}\||w_\varepsilon|^{3/2}\|^2+C\||w_\varepsilon|^{3/2}\|^2_{H^1}
\le C \|a\|_{L^{5/2}}^{5/2} \||w_\varepsilon|^{3/2}\|^2
+C\|g\|_{L^{5/3}}^{5/3} 
(\||w_\varepsilon|^{3/2}\|^2)^{4/9}.
\end{equation}
 
 \noindent\underline{\it Step 2:} Estimates for $|\nabla w_\varepsilon|^{3/2}$.

 By testing \eqref{lem-C-2} by $-\nabla\cdot(|\nabla w_\varepsilon|\nabla w_\varepsilon)\in L^{5/3 }(Q)$ we can obtain the following estimate (see the proof in Appendix B below):
 \begin{eqnarray}
&&\!\frac13 \frac{d}{dt} \||\nabla w_\varepsilon|^{3/2}\|^2+\frac49\|\nabla(|\nabla w_\varepsilon|^{3/2})\|^2
 \!+\int_\Omega|\nabla w_\varepsilon||D^2w_\varepsilon|^2\nonumber\\
 && \le \!\delta\||\nabla w_\varepsilon|^{3/2}\|^2_{H^1}\!+\!C\||\nabla w_\varepsilon|^{3/2}\|^2\!+\!\!\int_\Omega|a w_\varepsilon\nabla\!\cdot\!(|\nabla w_\varepsilon|\nabla w_\varepsilon)|\!+\!\!\int_\Omega|g\nabla\!\cdot\!(|\nabla w_\varepsilon|\nabla w_\varepsilon)|.
\label{lem-C-4}
 \end{eqnarray}
 Moreover, taking into account that
 $\nabla\cdot(|\nabla w_\varepsilon|\nabla w_\varepsilon)=|\nabla w_\varepsilon|\Delta w_\varepsilon+\frac23|\nabla w_\varepsilon|^{-1/2}\partial_j(|\nabla w_\varepsilon|^{3/2})\partial_jw_\varepsilon,$ from the H\"older  and 
 Young inequalities and considering the inequality 
 $\|u\|_{L^{10/3}}\le C\|u\|^{2/5}\|u\|^{3/5}_{H^1}$,  we can estimate as 
  \begin{eqnarray}
 &&\int_\Omega|g\nabla\cdot(|\nabla w_\varepsilon|\nabla w_\varepsilon)|
 \nonumber\\
 &&\le\int_\Omega |g| \, |\nabla w_\varepsilon|\, |\Delta w_\varepsilon|
+\frac23\int_\Omega |g| \, |\nabla w_\varepsilon|^{-1/2} |\nabla(|\nabla w_\varepsilon|^{3/2})\cdot\nabla w_\varepsilon|
\nonumber\\
&& \le \|g\|_{L^{5/2}}\left(\int_\Omega|\nabla w_\varepsilon||\Delta w_\varepsilon|^2\right)^{1/2}
 \||\nabla w_\varepsilon|^{3/2}\|^{1/3}_{L^{10/3}}
 +\frac23 \|g\|_{L^{5/2}} \||\nabla w_\varepsilon|^{3/2}\|^{1/3}_{L^{10/3}}\|\nabla(|\nabla w_\varepsilon|^{3/2})\|\nonumber\\
 &&\le\delta\left(\||\nabla w_\varepsilon|^{3/2}\|^2_{H^1}+\int_\Omega|\nabla w_\varepsilon||\Delta w_\varepsilon|^2\right)
 +C\|g\|_{L^{5/2}}^{2} \||\nabla w_\varepsilon|^{3/2}\|^{4/15} 
 \||\nabla w_\varepsilon|^{3/2}\|_{H^1}^{2/5}
 \nonumber\\
 &&\le\delta\left(\||\nabla w_\varepsilon|^{3/2}\|^2_{H^1}+\int_\Omega|\nabla w_\varepsilon||\Delta w_\varepsilon|^2\right)
 +C\|g\|_{L^{5/2}}^{5/2} (\||\nabla w_\varepsilon|^{3/2}\|^2)^{1/6} ,
  \label{lem-C-5-1}
 \end{eqnarray}
  for any $\delta>0$.
Reasoning in a similar way, and using the inequalities
$$
\|a w_\varepsilon\|_{L^{5/2}}
\le \|a\|_{L^{5/2+}}\|w_\varepsilon\|_{L^{\infty-}}
\le C \|a\|_{L^{5/2+}}\|w_\varepsilon\|_{W^{1,3}}
$$
and $$
\|w_\varepsilon\|_{W^{1,3}}^3
=
\||\nabla w_\varepsilon|^{3/2}\|^{2}+\| |w_\varepsilon|^{3/2}\|^{2},
$$
 it is possible to obtain
 \hspace{-0.3cm}\begin{eqnarray}
&&\int_\Omega|a w_\varepsilon\nabla\cdot(|\nabla w_\varepsilon|\nabla w_\varepsilon)|
 \nonumber\\
&&\ \ \le \delta\left(\!\||\nabla w_\varepsilon|^{3/2}\|^2_{H^1}\!+\!\int_\Omega|\nabla w_\varepsilon||\Delta w_\varepsilon|^2\!\right) 
\!+\!C\|a\|^{5/2}_{L^{5/2+}}
\left(\!
\||\nabla w_\varepsilon|^{3/2}\|^2 \!+\! \||w_\varepsilon|^{3/2}\|^2
\!\right).
\label{lem-C-5}
 \end{eqnarray}

  \noindent\underline{\it Step 3:}  
 By adding \eqref{lem-C-3} and \eqref{lem-C-4}, considering \eqref{lem-C-5-1}-\eqref{lem-C-5}, and choosing $\delta>0$ small enough, there exists positive constants $C_1,C_2,C_3,C_4$ such that
 \begin{eqnarray}
&& \frac{d}{dt}\left(\||w_\varepsilon|^{3/2}\|^2+\||\nabla w_\varepsilon|^{3/2}\|^2\right)
 +C_1\left(\||w_\varepsilon|^{3/2}\|^2_{H^1}+\||\nabla w_\varepsilon|^{3/2}\|^2_{H^1}\right)\nonumber\\
&&
\ \ \le {C}_2\left(\|a\|^{5/2}_{L^{5/2}}+\|g\|_{L^{5/3}}^{5/3}
+
 \|a\|^{5/2}_{L^{5/2+}}
 \right)\||w_\varepsilon|^{3/2}\|^2
 +{C}_3\left(\|g\|^{5/3}_{L^{5/3}}+\|g\|^{5/2}_{L^{5/2}}\right)
 \nonumber\\
 &&
 \ \ +{C}_4\left(\|g\|^{5/2}_{L^{5/2}} + \|a\|^{5/2}_{L^{5/2+}} \right)
 \||\nabla w_\varepsilon|^{3/2}\|^2 .
\label{lem-C-6}
 \end{eqnarray}
 Then, applying the Gronwall Lemma in \eqref{lem-C-6}, since $w_\varepsilon^0$ is bounded in  $W^{1,3}(\Omega)$,  we deduce that the sequences $\{|w_\varepsilon|^{3/2}\}_{\varepsilon>0}$ and $\{|\nabla w_\varepsilon|^{3/2}\}_{\varepsilon>0}$ are bounded in $W_2$. 
 Passing to the limit as $\varepsilon$ goes to 0 we deduce that $w_\varepsilon$ converges to the (unique) solution $w$ of problem \eqref{lem-C-1} with  
 $|w|^{3/2}$ and $|\nabla w|^{3/2}\in W_2$.
 \end{proof}
 
 \begin{lemma}\label{corol-A}
Let 
$\Omega\subset\mathbb{R}^3$ be a bounded domain with boundary $\partial\Omega\in C^2$. 
Assume that 
$ a\in L^{5/2}(Q)$,
${\bf c}\in L^5(Q)^3$, 
$d\in L^5(Q)$, 
$g^0_U\in L^{5\alpha(p)/(2\alpha(p)+3)} (Q)$,
${\bf g}^1\in L^{5\alpha(p)/(\alpha(p)+3)}(Q)^3$, 
 $\beta_1\in L^{5/2}(L^{5/2+})$,  
 $\beta_2\in L^{5}(Q)$  and  $g_V\in L^{5/2}(Q)$.  
 Then the coupled linear parabolic system  
\begin{equation}\label{corol-A-1}
\left\{
\begin{array}{rcl}\partial_tU-\Delta U+aU+\nabla\cdot(U{\bf c})+\nabla\cdot(d\nabla V)&=&g^0_U-\nabla\cdot{\bf g}^1\ \mbox{ in } Q,\\
\partial_tV-\Delta V+\beta_1V+\beta_2U&=&g_V\ \mbox{ in } Q,\\
U(0)&=&0,\ V(0)=0\ \mbox{ in } \Omega,\\
(-\nabla U+U{\bf c}+d\nabla V +{\bf g}^1)\cdot{\bf n}&=&0\ \mbox{ on } (0,T)\times\partial\Omega,\\
\partial_{\bf n}V&=&0\ \mbox{ on } (0,T)\times\partial\Omega,
\end{array}
\right.
\end{equation}
has a unique weak solution $(U,V)\in X_u\times X_v$.

\end{lemma}
\begin{proof}
By arguing again by regularization of data, the key will be find (formally) a priori estimates of 
$(|U|^{3/2},|V|^{3/2},|\nabla V|^{3/2})\in W_2\times W_2 \times W_2$.

By applying the proof of Lemma \ref{lem-B} to (\ref{corol-A-1})$_1$ for $w=U$ and ${\bf g}^1=d \nabla V$ (accounting in particular estimate \eqref{lem-B-8} for $q=3$)  and using the inequality 
$$\|d \nabla V\|_{L^{5/2}}
\le \|d\|_{L^{5}} \|\nabla V\|_{L^5}
= \|d\|_{L^{5}} \||\nabla V|^{3/2}\|_{L^{10/3}}^{2/3},
$$ 
one has
\begin{eqnarray}
&&\frac{d}{dt}\||U|^{ 3/2}\|^2 +
C\||U|^{3/2}\|^2_{H^1}
\nonumber\\
&&\ \ \le C\left(\| a\|^{5/2}_{L^{5/2}}+\|{\bf c}\|^5_{L^5}
+
 \|g^0_U\|_{L^{5/3}}^{5/3}
+
 \|{\bf g}^1\|^{5/2}_{L^{5/2}}
\right)
\||U|^{3/2}\|^2. 
 \label{estim-d-bis}
\\
&& \nonumber
\ \ +C\left(
\|g^0_U\|_{L^{5/3}}^{5/3}
+
\|{\bf g}^1\|^{5/2}_{L^{5/2}}
\right)
+
\|d\|_{L^{5}}
\||\nabla V|^{3/2}\|_{L^{10/3}}^{2/3}
\||U|^{3/2}\|_{L ^{10/3}}^{1/3} \||U|^{3/2}\|_{H^1}.
\label{corol-A-3-1-bis}
\end{eqnarray}
For simplicity, we introduce the vectorial variable 
${\bf Y}=(
|U|^{3/2},
|V|^{3/2},|\nabla V|^{3/2})$, hence the last term of \eqref{estim-d-bis} is bounded by 
\begin{equation} \label{estim-d}
C \|d\|_{L^{5}} 
\| {\bf Y} \|_{L ^{10/3}} \| {\bf Y} \|_{H^1}
\le C \|d\|_{L^{5}} \| {\bf Y} \|^{2/5} \| {\bf Y} \|_{H^1}^{8/5}
\le \delta  \| {\bf Y} \|_{H^1}^{2} + C \|d\|_{L^{5}}^5 \| {\bf Y} \|^{2}.
\end{equation}
On the other hand, testing $V$-equation by $|V|V-\nabla\cdot(|\nabla V|\nabla V)$ and arguing as in  the proof of Lemma \ref{lem-C} (see \eqref{lem-C-3} and \eqref{lem-C-5-1}
 for $w=V$ and $g=g_V-\beta_2 U$) we arrive at 

\begin{eqnarray}\label{corol-A-6}
&&\frac{d}{dt}\left(\||V|^{3/2}\|^2+\||\nabla V|^{3/2}\|^2\right)
+C\left(\||V|^{3/2}\|_{H^1}^2+\||\nabla V|^{3/2}\|^2_{H^1}\right)+C\int_\Omega|\nabla V||D^2 V|^2\nonumber\\
&&\ \ \le C\left(\| a\|^{5/2}_{L^{5/2}}+\|g_V\|^{5/3}_{L^{5/3}}
+\| \beta_1\|^{5/2}_{L^{5/2+}}
\right)\||V|^{3/2}\|^2
+C\left(\|\beta_1\|^{5/2}_{L^{5/2+}}+\|g_V\|^{5/2}_{L^{5/2}}
\right)\||\nabla V|^{3/2}\|^2\nonumber\\
&&\  \ +C\left(\|g_V\|^{5/3}_{L^{5/3}}+\|g_V\|^{5/2}_{L^{5/2}}\right)
+\int_\Omega |\beta_2|\, |U|\,|V|\,|V|+\int_\Omega|\beta_2U\nabla\cdot(|\nabla V|\nabla V)|.
\end{eqnarray}

We bound the last two terms as follows:
 \begin{eqnarray}
\int_\Omega|\beta_2|\, |U|\,|V|\,|V|
&\le&\int_\Omega |\beta_2| (|U|^{3/2})^{2/3} (|V| ^{3/2})^{4/3}
\le  \Vert \beta_2 \Vert_{L^{5/2}} \|{\bf Y}\|_{L^{10/3}}^2
\nonumber\\
&\le& C \Vert \beta_2 \Vert_{L^{5/2}} \|{\bf Y}\|^{4/5}  \|{\bf Y}\|_{H^1}^{6/5}
\le  \delta  \|{\bf Y}\|_{H^1}^{2} + C \Vert \beta_2 \Vert_{L^{5/2}}^{5/2} \|{\bf Y}\|^{2}.
\label{corol-A-8}
\end{eqnarray}
Arguing as in \eqref{lem-C-5-1} and using the bound
$$\|\beta_2U\|_{L^{5/2}}
\le \|\beta_2\|_{L^{5}} \| U\|_{L^5}
= \|\beta_2\|_{L^{5}} \| |U|^{3/2} \|_{L^{10/3}}^{2/3}
\le C \|\beta_2\|_{L^{5}} \| |U|^{3/2} \|^{4/15}  \| |U|^{3/2} \|_{H^1}^{2/5},
$$
  we have 
\begin{eqnarray}
\int_\Omega| \beta_2U\nabla\cdot(|\nabla V|\nabla V)|
&\le& \delta\left(\int_\Omega |\nabla V||\Delta V|^2+\||\nabla V|^{3/2}\|^2_{H^1}+\||U|^{3/2}\|^2_{H^1}\right) 
\nonumber\\
&&\ \ +C\|\beta_2\|_{L^{5}}^{5/2} \| |U|^{3/2} \|^{2/3}  \| |U|^{3/2} \|_{H^1} \||\nabla V|^{3/2}\|^{1/3}  
\nonumber\\
&&\le \delta\left(\int_\Omega |\nabla V||\Delta V|^2+\||\nabla V|^{3/2}\|^2_{H^1}+\||U|^{3/2}\|^2_{H^1}\right)\nonumber\\ 
&&\ \ +C\|\beta_2\|_{L^{5}}^{5} \|{\bf Y} \|^{2}. 
\label{corol-A-9}
\end{eqnarray}

Therefore, adding \eqref{estim-d-bis} and \eqref{corol-A-6} and taking into account \eqref{estim-d} and \eqref{corol-A-8}-\eqref{corol-A-9} we obtain the inequality 
\begin{equation}
\frac{d}{dt}\|{\bf Y} \|^2 + C\|{\bf Y} \|^2_{H^1}
\le k(t) \|{\bf Y} \|^2
+C\left( \|g^0_U\|_{L^{5/3}}^{5/3} +
\|{\bf g}^1\|^{5/2}_{L^{5/2}} \right),
\label{corol-A-10}
\end{equation}
with
$$
k(t)=
C\left(\| a\|^{5/2}_{L^{5/2}}+\|{\bf c}\|^5_{L^5}
+ \|g^0_U\|_{L^{5/3}}^{5/3}
+ \|{\bf g}^1\|^{5/2}_{L^{5/2}}
+\| d \|^5_{L^5}
+\|\beta_1\|^{5/2}_{L^{5/2+}}+\|g_V\|^{5/2}_{L^{5/2}}
+\|\beta_2\|^5_{L^5}
\right).
$$
Therefore, from \eqref{corol-A-10} and the Gronwall lemma we deduce that 
${\bf Y}=(|U|^{3/2},|V|^{3/2},|\nabla V|^{3/2})$ is bounded in $W_2\times W_2\times W_2$. Since in particular $\nabla V$ is bounded in $L^5(Q)$ and $d\in L^{5}(Q)$, then  $d \nabla V$ is bounded in $L^{5/2}(Q)$. Consequently, by applying again Lemma \ref{lem-B} (with $q=\alpha(p)$) to (\ref{corol-A-1})$_1$, $U$ is bounded in $L^\infty(L^{\alpha(p)})\cap L^{\beta(p)}(Q)$  and $\nabla U$ in $L^2(Q)$, hence $U\in X_u$. On the other hand, since $V(0)=0$ and $g_V\in L^{5/2}(Q)$, from parabolic regularity we deduce that $V\in X_{5/2}$.  
Finally, the uniqueness can be obtained following a classical comparison argument.
\end{proof}
\section{Proof of Estimate \eqref{lem-C-4}}

 Testing \eqref{lem-C-2}$_1$ by $-\nabla\cdot(|\nabla w_\varepsilon|\nabla w_\varepsilon)\in L^{5/2}(Q)$ we have
\begin{eqnarray}
&&\frac13\frac{d}{dt}\||\nabla w_\varepsilon|^{3/2}\|^2+\int_\Omega\Delta w_\varepsilon\nabla\cdot(|\nabla w_\varepsilon|\nabla w_\varepsilon)\nonumber\\
&&\ \ \ \ \ \le \int_\Omega|a_\varepsilon w_\varepsilon\nabla\cdot(|\nabla w_\varepsilon|\nabla w_\varepsilon)|
+\int_\Omega|g^0_\varepsilon\nabla\cdot(|\nabla w_\varepsilon|\nabla w_\varepsilon)|.\label{AP-1}
\end{eqnarray}
To prove inequality \eqref{lem-C-4} it is suffices to control the term involving $\Delta w_\varepsilon$. Indeed, using the Einstein summation convection we have
\begin{eqnarray*}
\int_\Omega\Delta w_\varepsilon\nabla\cdot(|\nabla w_\varepsilon|\nabla w_\varepsilon)
&=&\int_\Omega\partial^2_{ll}w_\varepsilon\partial_j(|\nabla w_\varepsilon|\partial_j w_\varepsilon)
=-\int_\Omega\partial_l(\partial_{lj}^2w_\varepsilon)(|\nabla w_\varepsilon|\partial_jw_\varepsilon)\\
&=&\int_\Omega\partial^2_{lj}w_\varepsilon\partial_l(|\nabla w_\varepsilon|\partial_jw_\varepsilon)
-\int_{\partial\Omega}\partial^2_{lj}w_\varepsilon|\nabla w_\varepsilon|\partial_jw_\varepsilon n_l.
\end{eqnarray*}
Notice that 
\begin{eqnarray*}
\partial_l(|\nabla w_\varepsilon|\partial_jw_\varepsilon)
&=&|\nabla w_\varepsilon|\partial^2_{jl}w_\varepsilon+|\nabla w_\varepsilon|^{-1}\nabla w_\varepsilon\partial_l(\nabla w_\varepsilon)\partial_jw_\varepsilon,\\
\partial_l(|\nabla w_\varepsilon|^{3/2})&=&\frac32|\nabla w_\varepsilon|^{-1/2}\nabla w_\varepsilon\cdot\partial_l(\nabla w_\varepsilon).
\end{eqnarray*}
Thus,
\begin{eqnarray*}
\partial^2_{lj}w_\varepsilon\partial_l(|\nabla w_\varepsilon|\partial_jw_\varepsilon)
&=&|\nabla w_\varepsilon||D^2w_\varepsilon|^2+|\nabla w_\varepsilon|^{-1/2}\nabla w_\varepsilon\cdot\partial_l(\nabla w_\varepsilon)|\nabla w_\varepsilon|^{-1/2}\partial_jw_\varepsilon\partial_l(\partial_jw_\varepsilon)\\
&=&|\nabla w_\varepsilon||D^2w_\varepsilon|^2+\frac49|\nabla(|\nabla w_\varepsilon|^{3/2})|^2.
\end{eqnarray*}
Consequently,
\begin{eqnarray}
\int_\Omega \Delta w_\varepsilon\nabla\cdot(|\nabla w_\varepsilon|\nabla w_\varepsilon)
&=&\int_\Omega|\nabla w_\varepsilon||D^2w_\varepsilon|^2+\frac49\|\nabla(|\nabla w_\varepsilon|^{3/2})\|^2\nonumber\\
&&-\int_{\partial\Omega}\partial^2_{lj}w_\varepsilon|\nabla w_\varepsilon|\partial_jw_\varepsilon n_l.\label{AP-2}
\end{eqnarray} 
For the analysis of the boundary term given in \eqref{AP-2}, we will use a technical result proved in \cite[Lemma 4.10]{Mizoguchi}, which establishes that if $z\in C^2(\Omega)$ with $\partial_{\bf n}z|_{\partial\Omega}=0$, then $\partial_{\bf n}|\nabla z|^2\le \mathcal{K}|\nabla z|^2$, where $\mathcal{K}$ is an upper bound on the curvature of $\partial\Omega$. Thus, noting that $\partial^2_{lj}w_\varepsilon\partial_jw_\varepsilon n_l=\frac12\partial_{\bf n}|\nabla w_\varepsilon|^2$, we have

\begin{eqnarray}
\int_\Omega\partial^2_{lj}w_\varepsilon|\nabla w_\varepsilon|\partial_jw_\varepsilon n_l
&=&\frac{\mathcal{K}}{2}\int_{\partial\Omega}|\nabla w_\varepsilon|^2=\frac{\mathcal{K}}{2}\||\nabla w_\varepsilon|^{3/2}\|^2_{L^2(\partial\Omega)}
\le C\||\nabla w_\varepsilon|^{3/2}\|^2_{H^{1/2+\varepsilon}}\nonumber\\
&\le&\delta\|\nabla(|\nabla w_\varepsilon|^{3/2})\|^2+C_\delta\||\nabla w_\varepsilon|^{3/2}\|^2,\label{AP-3}
\end{eqnarray}
for $\delta>0$ arbitrarily small. Therefore, from \eqref{AP-1}-\eqref{AP-3} we deduce estimate \eqref{lem-C-4}.

\section*{Acknowledgements}
F. Guill\'en-Gonz\'alez and M.A. Rodr\'iguez-Bellido have been partially supported by Grant  I+D+I PID2023-149182NB-I00 funded by MICIU/AEI/10.13039/501100011033 and, 
ERDF/EU.  
E. Mallea-Zepeda has been supported by ANID-Chile through of project research
Fondecyt de Iniciaci\'on 11200208. E.J. Villamizar-Roa has been supported by Vicerrector\'ia de Investigaci\'on y Extensi\'on of the Universidad Industrial de Santander, Project 4210.

\end{document}